\documentclass[12pt]{article}
\usepackage{graphicx}
\usepackage{tabularx}
\usepackage{amsmath,amsthm}
\usepackage{amssymb}
\usepackage{amsbsy}
\usepackage{latexsym, amsfonts}
\usepackage{graphicx}
\usepackage{tikz}
\usepackage{enumerate}

\usepackage{hyperref}

\usepackage[utf8]{inputenc}

\newcommand{\A}	{\includegraphics{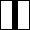}}
\newcommand{\Aou}{\includegraphics{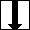}}
\newcommand{\Auo}{\includegraphics{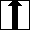}}
\newcommand{\B}	{\includegraphics{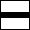}}
\newcommand{\Blr}	{\includegraphics{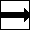}}
\newcommand{\Brl}	{\includegraphics{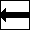}}
\newcommand{\C}	{\includegraphics{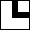}}
\newcommand{\Cor}	{\includegraphics{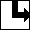}}
\newcommand{\Cro}	{\includegraphics{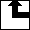}}
\newcommand{\D}	{\includegraphics{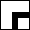}}
\newcommand{\Dru}	{\includegraphics{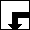}}
\newcommand{\Dur}	{\includegraphics{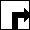}}
\newcommand{\E}	{\includegraphics{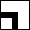}}
\newcommand{\Eul}	{\includegraphics{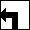}}
\newcommand{\Elu}	{\includegraphics{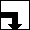}}
\newcommand{\F}	{\includegraphics{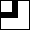}}
\newcommand{\Flo}	{\includegraphics{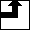}}
\newcommand{\Fol}	{\includegraphics{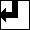}}
\newcommand{\CE}{\C\hspace{-0.5pt}\E}
\newcommand{\CEou}{\C\hspace{-0.5pt}\Elu}
\newcommand{\DF}{\D\hspace{-1pt}\F}
\newcommand{\p} {\mathrm{path}}
\newcommand{\ty} {\mathrm{type}}

\newcommand{\ebottom}{\operatorname{bottom}}
\newcommand{\eleft}{\operatorname{left}}
\newcommand{\eright}{\operatorname{right}}
\newcommand{\etop}{\operatorname{top}}

\theoremstyle{definition}
\newtheorem{definition}{Definition}
\newtheorem{remark}{Remark}

\newtheorem{assumptions}{Assumptions}
\newtheorem*{notation*}{Notation}
\theoremstyle{plain}
\newtheorem{theorem}{Theorem}
\newtheorem{lemma}{Lemma}

\newtheorem{proposition}{Proposition}
\newtheorem{conjecture}{Conjecture}
\newtheorem{corollary}{Corollary}

\definecolor{magenta}{rgb}{0.75,0,0.25}

\definecolor{violet}{rgb}{0.25,0,0.75}

\definecolor{darkgreen}{rgb}{0.0,0.5,0.0}

\newcommand{\fr}{\mathrm{fr}}

\title{Supermixed labyrinth fractals}

\author{Ligia L. Cristea \thanks{L.L. Cristea is supported by the Austrian Science Fund (FWF), 
stand-alone project P27050-N26, and by the Austrian Science Fund (FWF) project F5508-N26, which is part of the Special Research Program ``Quasi-Monte Carlo Methods: Theory and Applications''}  \\
 \tt{strublistea@gmail.com} 
\and Gunther Leobacher \thanks{
G. Leobacher is supported by the Austrian Science Fund (FWF) project F5508-N26, which is part of the Special Research Program ``Quasi-Monte Carlo Methods: Theory and Applications''
}\\  \tt{gunther.leobacher@uni-graz.at} \\\\Karl-Franzens-Universit\"at Graz\\ Institut f\"ur Mathematik und Wissenschaftliches Rechnen
\\Heinrichstrasse 36, 8010 Graz\\Austria\\}

\begin{document}

\maketitle


\abstract{Labyrinth fractals are dendrites in the unit square. They were
introduced and studied in the last decade first in the self-similar case
\cite{laby_4x4,laby_oigemoan}, then in the mixed case
\cite{mixlaby,cristealeobacher_arcs}. Supermixed fractals constitute  a 
significant generalisation
of mixed labyrinth fractals: each step of the iterative construction is done
according to not just one labyrinth pattern, but possibly to several
different patterns. In this paper we introduce and study supermixed labyrinth fractals and the
corresponding prefractals, called supermixed labyrinth sets, with focus on
the aspects that were previously studied for the self-similar and mixed case:
topological properties and properties of the arcs between points in the
fractal. The facts and formul{\ae} found here extend results proven in the
above mentioned cases. One of the main results is a sufficient condition
for infinite length of arcs in mixed labyrinth fractals.}

\section{Introduction}

Labyrinth fractals are a special family of Sierpi\'nski carpets in the plane,
and are dendrites in the unit square. They were defined and studied in the last
decade, first in the self-similar case \cite{laby_4x4, laby_oigemoan}, then in
the more general case, as mixed labyrinth fractals \cite{mixlaby,
cristealeobacher_arcs}. The self-similar labyrinth fractals are generated by
one labyrinth pattern, while in the case of mixed labyrinth fractals a sequence
of patterns may be used, i.e., one labyrinth pattern for each step of the iterative
construction. In both cases, in order to study the arcs between points in the
fractals, one uses the paths in the graphs (more precisely, trees) associated
to the patterns and to the prefractals, called labyrinth sets of some level
$n\ge 1$, respectively. An important role is played by the path matrix
associated to a labyrinth pattern and, respectively, to a labyrinth set of
level $n$. 

The results on labyrinth fractals have  already found applications in physics,
in different contexts like, e.g., the study of planar nanostructures
\cite{GrachevPotapovGerman2013}, the fractal reconstruction of complicated
images, signals and radar backgrounds \cite{PotapovGermanGrachev2013}, and
recently in the construction of prototypes of ultra-wide band radar antennas
\cite{PotapovZhang2016}. Moreover, in very recent work
\cite{PotapovPotapovPotapov_dec2017} fractal labyrinths are used in combination
with genetic algorithms for the synthesis of big robust antenna arrays and
nano-antennas in 
telecommunication.  Let us remark here that supermixed labyrinth fractals, as
well as self-similar and mixed labyrinth fractals, are  so-called finitely
ramified carpets, in the terminology used  in physicists' work regarding the
modelling of porous structures  \cite{Tarafdar_modelporstructrepeatedSC2001}.
Moreover, physicists use so-called disordered fractals to study diffusion in disordered media \cite{AnhHoffmanSeegerTarafdar2005}, more precisely, they investigate the diffusion in fractals obtained by mixing different Sierpi\'nski carpet generators. This leads to the idea that supermixed labyrinth fractals are a class of such fractals that can be used as models for future research.
Here we also remark that random Koch curves are related, e.g., to arcs between
exits in supermixed labyrinth fractals, and are of interest to
theoretical physicists in the context of diffusion processes,  e.g.,
\cite{SeegerHoffmannEssex2009_randomKoch}. 
Since labyrinth fractals are dendrites, let us mention here that very recent research in materials engineering \cite{JanaGarcia_lithiumdendrite2017} shows that dendrite growth, a largely unsolved problem,  plays an essential role when dealing with high power and energy lithium-ion batteries. In the context of fractal dendrites there is also recent work \cite{tarafdar_multifractalNaCl2013} in crystal growth which lead us to the conclusion that certain families of our labyrinth fractals are suitable as models for such phenomena.

While the self-similar and mixed labyrinth fractals are appealing because of 
their elementary setup and their elegant algebraic treatment, they do not
always provide convincing models for applications. This is overcome by the
introduction of supermixed labyrinth fractals:
for those we have, at each step of the
iterative construction, not just one labyrinth pattern, but in general,
a finite collection of labyrinth patterns, according to which the construction
of the prefractal is done. Again, the graph of each of the resulting
prefractals is a
tree, and the supermixed labyrinth fractal is a dendrite, thus the paths
between given vertices,  and, respectively, the arcs that connect points in the
fractal are unique. Here, in order to obtain the length of paths in the
prefractals, the path matrix does not work in the same way as in the
mixed and self-similar case. Here we need to introduce  {\em counting
matrices}, specific to the supermixed case. In this more general case the
relations that hold when passing from one iteration to the next one are not
anymore ``encoded" by the powers of a matrix or by products of matrices, as in
the self-similar or mixed case, respectively. In the supermixed case the
formul{\ae} contain sums of products of path matrices of the patterns and
counting matrices associated to the iterations. 

Of course, one could also approach supermixed labyrinth fractals as  $V$-variable fractals, see, e.g., \cite{barnsley_superfractals,freiberghamblyhutchinson_Vvariable} , or in the context of graph directed constructions, see, e.g., \cite{MauldinWilliams_graphdirected_1988}, or that of graph directed Markov systems
\cite{MauldinUrbanski_bookGDMS_2003}, but here the idea was to remain in the same framework as in the  case of self-similar and mixed labyrinth fractals, the objects that we generalise here. 

Finally, we mention that there is very recent research on fractal dendrites \cite{Samuel_selfsimilar_dendrites}, where self-similar  dendrites are constructed by using polygonal systems in the plane, a method based on IFSs that is different from the construction method used for self-similar labyrinth fractals \cite{laby_4x4, laby_oigemoan}.

Let us now give a short outline of the paper. 
First, we recall notions about labyrinth fractals and introduce the concepts of 
supermixed labyrinth set and fractal in Section \ref{sec:notions}. We also
prove that supermixed labyrinth sets are labyrinth patterns.

In Section \ref{sec:topological} we prove that every supermixed labyrinth 
fractal is a dendrite.
Next we define the exits of supermixed labyrinth fractals and of squares of
a given level in the short Section \ref{sec:exits}.

Next in Section \ref{sec:path_matrices} we describe how paths in the
graphs of prefractals of supermixed labyrinth fractals are constructed
iteratively. We recall the definition of the path matrix of a labyrinth 
pattern and define counting matrices, a concept specific to supermixed
labyrinth fractals. The main result of this section is Theorem 
\ref{theo:basisformel}, which gives a recursive formula for the path matrices
of supermixed labyrinth sets of different levels.

Section \ref{sec:arcs_supermixed} is devoted to constructing, and exploring
properties of arcs in supermixed labyrinth fractals.

The concept of blocked labyrinth pattern is recalled in Section 
\ref{sec:blocked}.  We review existing results about 
arcs in self-similar and mixed labyrinth fractals constructed with
blocked patterns. Moreover, we recall properties of the path matrix of 
blocked labyrinth patterns.

Finally, we formulate and prove one of our main results in Section 
\ref{sec:arcs_mixed}: a sufficient condition for infinite length of any arc
between distinct points of a mixed labyrinth fractal is $\sum_{k=1}^\infty
\frac{1}{m_k}=\infty$, where $(m_k)_{k\ge 1}$ is the sequence of widths of
of the patterns that  define the fractal.
We also remark on difficulties in adapting the proof method to the supermixed
case.

\section{Patterns, labyrinth patterns, supermixed labyrinth sets and supermixed labyrinth fractals}\label{sec:notions}
In order to construct labyrinth fractals we use \emph{labyrinth patterns}.
Figures~\ref{fig:A1tildeA2}, and \ref{fig:W2} show labyrinth patterns and illustrate the first two steps of the 
construction described below.

Let $x,y,q\in [0,1]$ such that $Q=[x,x+q]\times [y,y+q]\subseteq [0,1]\times [0,1]$. 
Then for any point $(z_x,z_y)\in[0,1]\times [0,1]$ we define the function
\[
P_Q(z_x,z_y)=(q z_x+x,q z_y+y).
\]
Let $m\ge 1$. $S_{i,j,m}=\{(x,y)\mid \frac{i}{m}\le x \le \frac{i+1}{m} \mbox{ and } \frac{j}{m}\le y \le \frac{j+1}{m} \}$ and  
${\cal S}_m=\{S_{i,j,m}\mid 0\le i\le m-1 \mbox{ and } 0\le j\le m-1 \}$. 
\\
We call any nonempty ${\cal A} \subseteq {\cal S}_m$ an $m$-\emph{pattern} and $m$ its \emph{width}.\\  

Let $\{\widetilde{\cal A}_k\}_{k=1}^{\infty}$, with $\widetilde{\mathcal A}_k=\{ {\mathcal A}_{k,h},  \mbox{ for }h=1,\dots,s_k \}$, where $s_k\ge 1$ (the number of patterns of the collection $\widetilde{\mathcal A}_k$), for all $k\ge 1,$
be a sequence of nonempty collections of non-empty patterns and $\{m_k\}_{k=1}^{\infty}$ be the corresponding 
\emph{width-sequence}, i.e., for any $k\ge 1$ we have 
${\cal A}\subseteq {\cal S}_{m_k}$, for all ${\cal A}\in \widetilde{\mathcal A}_{k}$.
Throughout this paper we assume $s_1=1.$
\begin{figure}[hhhh]
\begin{center}
\begin{tikzpicture}[scale=.35]
\draw[line width=2pt] (0,0) rectangle (10,10);
\draw[line width=2pt] (2.5, 0) -- (2.5,10);
\draw[line width=2pt] (5, 0) -- (5,10);
\draw[line width=2pt] (7.5, 0) -- (7.5,10);
\draw[line width=2pt] (0, 2.5) -- (10,2.5);
\draw[line width=2pt] (0, 5) -- (10,5);
\draw[line width=2pt] (0, 7.5) -- (10,7.5);
\filldraw[fill=black, draw=black] (2.5,0) rectangle (5, 2.5);
\filldraw[fill=black, draw=black] (5,0) rectangle (7.5, 2.5);
\filldraw[fill=black, draw=black] (7.5,0) rectangle (10, 2.5);
\filldraw[fill=black, draw=black] (5,2.5) rectangle (7.5, 5);
\filldraw[fill=black, draw=black] (0,5) rectangle (2.5, 7.5);
\filldraw[fill=black, draw=black] (5,7.5) rectangle (7.5, 10);
\filldraw[fill=black, draw=black] (7.5,7.5) rectangle (10, 10);
\draw[line width=2pt] (11,0) rectangle (21,10);
\draw[line width=2pt] (13.5, 0) -- (13.5,10);
\draw[line width=2pt] (16, 0) -- (16,10);
\draw[line width=2pt] (18.5, 0) -- (18.5,10);
\draw[line width=2pt] (11, 2.5) -- (21,2.5);
\draw[line width=2pt] (11, 5) -- (21,5);
\draw[line width=2pt] (11, 7.5) -- (21,7.5);
\filldraw[fill=black, draw=black] (11,0) rectangle (16, 5);
\filldraw[fill=black, draw=black] (11,7.5) rectangle (13.5, 10);
\filldraw[fill=black, draw=black] (16,5) rectangle (18.5, 7.5);
\filldraw[fill=black, draw=black] (18.5,0) rectangle (21, 2.5);
\draw[line width=2pt] (22,0) rectangle (32,10);
\draw[line width=2pt] (24.5, 0) -- (24.5,10);
\draw[line width=2pt] (27,0) -- (27,10);
\draw[line width=2pt] (29.5, 0) -- (29.5,10);
\draw[line width=2pt] (22, 2.5) -- (32,2.5);
\draw[line width=2pt] (22, 5) -- (32,5);
\draw[line width=2pt] (22, 7.5) -- (32,7.5);
\filldraw[fill=black, draw=black] (22,0) rectangle (24.5, 5);
\filldraw[fill=black, draw=black] (22,7.5) rectangle (24.5, 10);
\filldraw[fill=black, draw=black] (24.5,0) rectangle (27, 2.5);
\filldraw[fill=black, draw=black] (27,5) rectangle (29.5, 7.5);
\filldraw[fill=black, draw=black] (29.5,7.5) rectangle (32, 10);
\filldraw[fill=black, draw=black] (29.5,0) rectangle (32, 2.5);
\end{tikzpicture}
\caption{Three labyrinth patterns (all of width $4$), from left to right: the unique pattern ${\cal A}_{1,1} \in \widetilde {\cal A}_1$, followed by the (two) patterns ${\cal A}_{2,1}, {\cal A}_{2,2} \in \widetilde{\cal A}_2$} \label{fig:A1tildeA2}
\end{center}
\end{figure}
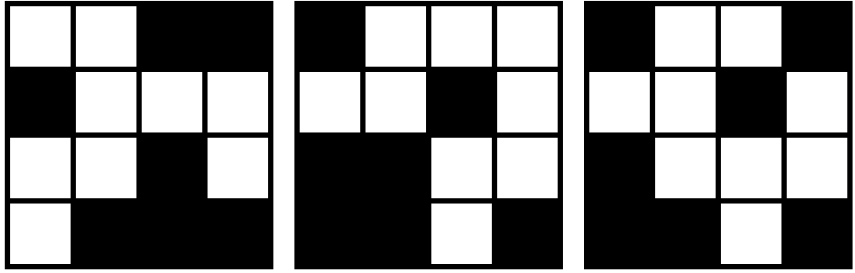

\begin{figure}[hhhh]\label{W2}
\begin{center}
\begin{tikzpicture}[scale=.3]
\draw[line width=1pt] (0,0) rectangle (16,16);
\draw[line width=0.8pt] (4, 0) -- (4,16);
\draw[line width=0.8pt] (8, 0) -- (8,16);
\draw[line width=0.8pt] (12, 0) -- (12,16);
\draw[line width=0.8pt] (0, 4) -- (16, 4);
\draw[line width=0.8pt] (0, 8) -- (16, 8);
\draw[line width=0.8pt] (0, 12) -- (16,12);
\draw[line width=0.5pt] (1, 0) -- (1,16);
\draw[line width=0.5pt] (2, 0) -- (2,16);
\draw[line width=0.5pt] (3, 0) -- (3,16);
\draw[line width=0.5pt] (5, 0) -- (5,16);
\draw[line width=0.5pt] (6, 0) -- (6,16);
\draw[line width=0.5pt] (7, 0) -- (7,16);
\draw[line width=0.5pt] (9, 0) -- (9,16);
\draw[line width=0.5pt] (10, 0) -- (10,16);
\draw[line width=0.5pt] (11, 0) -- (11,16);
\draw[line width=0.5pt] (13, 0) -- (13,16);
\draw[line width=0.5pt] (14, 0) -- (14,16);
\draw[line width=0.5pt] (15, 0) -- (15,16);
\draw[line width=0.5pt] (0, 1) -- (16,1);
\draw[line width=0.5pt] (0, 2) -- (16,2);
\draw[line width=0.5pt] (0, 3) -- (16,3);
\draw[line width=0.5pt] (0, 5) -- (16,5);
\draw[line width=0.5pt] (0, 6) -- (16,6);
\draw[line width=0.5pt] (0, 7) -- (16,7);
\draw[line width=0.5pt] (0, 9) -- (16,9);
\draw[line width=0.5pt] (0, 10) -- (16,10);
\draw[line width=0.5pt] (0, 11) -- (16,11);
\draw[line width=0.5pt] (0, 13) -- (16,13);
\draw[line width=0.5pt] (0, 14) -- (16,14);
\draw[line width=0.5pt] (0, 15) -- (16,15);
\filldraw[fill=black, draw=black] (0,8) rectangle (4, 12);
\filldraw[fill=black, draw=black] (4,0) rectangle (16, 4);
\filldraw[fill=black, draw=black] (8,4) rectangle (12, 8);
\filldraw[fill=black, draw=black] (8,12) rectangle (16, 16);
\filldraw[fill=black, draw=black] (0,0) rectangle (2, 2);
\filldraw[fill=black, draw=black] (0,3) rectangle (1, 4);
\filldraw[fill=black, draw=black] (2,2) rectangle (3, 3);
\filldraw[fill=black, draw=black] (3,0) rectangle (4, 1);
\filldraw[fill=black, draw=black] (0,4) rectangle (2, 6);
\filldraw[fill=black, draw=black] (0,7) rectangle (1, 8);
\filldraw[fill=black, draw=black] (2,6) rectangle (3, 7);
\filldraw[fill=black, draw=black] (3,4) rectangle (4, 5);
\filldraw[fill=black, draw=black] (4,8) rectangle (6, 10);
\filldraw[fill=black, draw=black] (4,11) rectangle (5, 12);
\filldraw[fill=black, draw=black] (6,10) rectangle (7, 11);
\filldraw[fill=black, draw=black] (7,8) rectangle (8, 9);
\filldraw[fill=black, draw=black] (4,12) rectangle (6, 14);
\filldraw[fill=black, draw=black] (4,15) rectangle (5, 16);
\filldraw[fill=black, draw=black] (6,14) rectangle (7, 15);
\filldraw[fill=black, draw=black] (7,12) rectangle (8, 13);
\filldraw[fill=black, draw=black] (12,4) rectangle (14, 6);
\filldraw[fill=black, draw=black] (12,7) rectangle (13, 8);
\filldraw[fill=black, draw=black] (14,6) rectangle (15, 7);
\filldraw[fill=black, draw=black] (15,4) rectangle (16, 5);
\filldraw[fill=black, draw=black] (0,12) rectangle (1, 14);
\filldraw[fill=black, draw=black] (0,15) rectangle (1, 16);
\filldraw[fill=black, draw=black] (1,12) rectangle (2, 13);
\filldraw[fill=black, draw=black] (2,14) rectangle (3, 15);
\filldraw[fill=black, draw=black] (3,12) rectangle (4, 13);
\filldraw[fill=black, draw=black] (3,15) rectangle (4, 16);
\filldraw[fill=black, draw=black] (4,4) rectangle (5, 6);
\filldraw[fill=black, draw=black] (4,7) rectangle (5,
 8);
\filldraw[fill=black, draw=black] (5,4) rectangle (6, 5);
\filldraw[fill=black, draw=black] (6,6) rectangle (7, 7);
\filldraw[fill=black, draw=black] (7,4) rectangle (8, 5);
\filldraw[fill=black, draw=black] (7,7) rectangle (8, 8);
\filldraw[fill=black, draw=black] (8,8) rectangle (9, 10);
\filldraw[fill=black, draw=black] (8,11) rectangle (9, 12);
\filldraw[fill=black, draw=black] (9,8) rectangle (10, 9);
\filldraw[fill=black, draw=black] (10,10) rectangle (11, 11);
\filldraw[fill=black, draw=black] (11,8) rectangle (12, 9);
\filldraw[fill=black, draw=black] (11,11) rectangle (12, 12);
\filldraw[fill=black, draw=black] (12,8) rectangle (13, 10);
\filldraw[fill=black, draw=black] (12,11) rectangle (13, 12);
\filldraw[fill=black, draw=black] (13, 8) rectangle (14, 9);
\filldraw[fill=black, draw=black] (14,10) rectangle (15, 11);
\filldraw[fill=black, draw=black] (15, 8) rectangle (16, 9);
\filldraw[fill=black, draw=black] (15,11) rectangle (16, 12);
\end{tikzpicture}
\caption{The set ${\cal W}_2$, constructed with the help of  the above patterns from $\widetilde{\cal A}_1$ and $\widetilde{\cal A}_2$ (with $m_1=m_2 = 4$), that
can also be viewed as a $16$-pattern} 
\label{fig:W2}
\end{center}
\end{figure}
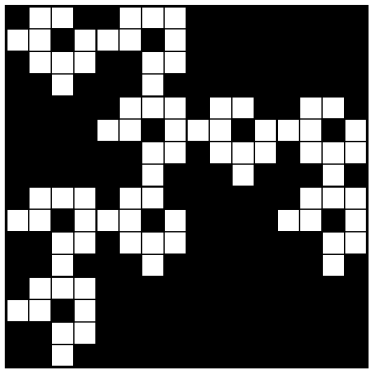

We denote $m(n)=\prod_{k=1}^n m_k$, for all $n \ge 1$. 
We let ${\cal W}_1:={\cal A}_{1,1}$, and call it the 
\emph{set of white squares of level $1$}. 
Then we define ${\cal B}_1={\cal S}_{m_1} \setminus {\cal W}_1$ 
as the \emph{set of black squares of level $1$}.
For $k\ge 1$ let $\phi_{k+1}:{\mathcal W}_k \rightarrow \widetilde{\mathcal A}_{k+1} $, i.e., $\phi_{k+1}$ assigns to every square $W\in {\mathcal W}_k $ a  pattern in $\widetilde{\mathcal A}_{k+1}$.
For $n\ge 2$ we define the \emph{set ${\mathcal W}_n$ of white squares of level $n$} as follows.

\begin{equation} \label{eq:W_n}
{\cal W}_n:=\bigcup_{W\in {\phi_n}(W_{n-1}), W_{n-1}\in {\cal W}_{n-1}}\{ P_{W_{n-1}}(W)\}, \mbox{ for  all }n \ge 2 .
\end{equation}

We note that ${\cal W}_n\subset {\cal S}_{m(n)}$, and we define the \emph{set of black squares of level $n$} by ${\cal B}_n={\cal S}_{m(n)} \setminus {\cal W}_n$. For $n\ge 1$, we define $L_n=\bigcup_{W\in {\cal W}_n} W$. 
One can immediately see that  $\{L_n\}_{n=1}^{\infty}$ is a monotonically decreasing sequence of compact sets. 
We call $L_{\infty}:=\bigcap_{n=1}^{\infty}L_n$ the \emph{limit set defined by the sequence of collections of patterns 
$\{{\cal A}_k\}_{k=1}^{\infty}.$ }                                                                                                                                                                                                                                                                                                                                                                                                                                                                                                                                                                                                                                                                                                                                                                                                                                                                                                                                                                                                                              
                                                                                                                                                                                                                                                                                                                                                                                                                                                                                                                                                                                                                                                                                                                                                                                                                                                                                    
\vspace{0.25cm}

A {\em graph} $\mathcal{G}$ is a pair $(\mathcal{V},\mathcal{E})$, where $\mathcal{V}=\mathcal{V}(\mathcal{G})$ is a finite set of vertices, and the set of edges $\mathcal{E}=\mathcal{E}(\mathcal{G})$ is a subset of $\{\{u,v\} \mid u,v \in \mathcal{V}, u\neq v\}$. We write $u \sim v$ if $\{u,v\}\in\mathcal{E}(\mathcal{G})$ and we say $u$ is a \emph{neighbour} of $v$. The sequence of vertices $\{u_i\}_{i=0}^{n}$ is a \emph{path between $u_0$ and $u_n$} in a graph $\mathcal{G}\equiv (\mathcal{V},\mathcal{E})$, if $u_0,u_1,\ldots,u_n\in \mathcal{V}$, $u_{i-1}\sim u_i$ for $1 \le i\le n$, and $u_i\neq u_j$ for $0\le i<j\le n$.  The sequence of vertices $\{u_i\}_{i=0}^{n}$ is a \emph{cycle} in $\mathcal{G}\equiv (\mathcal{V},\mathcal{E})$, if $u_0,u_1,\ldots,u_n\in \mathcal{V}$, $u_{i-1}\sim u_i$ for $1 \le i\le n$, $u_i\neq u_j$ for $1\le i<j\le n$, and $u_0=u_n$. A \emph{tree} 
is a connected graph that contains no cycle. For any two distinct vertices in a tree there exists a unique cycle-free path in the tree that connects them.

For ${\cal A}\subseteq {{\cal S}_m}$, we define $\mathcal{G}({\cal A})\equiv (\mathcal{V}(\mathcal{G}({\cal A})),\mathcal{E}(\mathcal{G}({\cal A})))$ to be the graph of ${\cal A}$, i.e., the graph whose vertices $\mathcal{V}(\mathcal{G}({\cal A}))$ are the white squares in ${\cal A}$, and whose edges $\mathcal{E}(\mathcal{G}({\cal A}))$ are the unordered pairs of white squares, that have a common side. The \emph{top row} in ${\cal A}$ is the set of all white squares in $\{S_{i,m-1,m}\mid 0\le i\le m-1 \}$. The bottom row, left column, and right column in ${\cal A}$  are defined analogously. A \emph{top exit} in ${\cal A}$ is a white square in the top row, such that there is a white square in the same column in the bottom row. A \emph{bottom exit} in ${\cal A}$  is defined analogously. A \emph{left exit} in ${\cal A}$ is a white square in the left column, such that there is a white square in the same row in the right column. A \emph{right exit} in ${\cal A}$ is defined analogously. 
While a top exit together with the corresponding bottom exit constitute a \emph{vertical exit pair}, a left exit and the corresponding right exit constitute a \emph{horizontal exit pair}.

We recall the definition of a labyrinth pattern, see \cite[page 3]{laby_4x4}.

\begin{definition}
A non-empty $m\times m$-pattern ${\cal A} \subseteq {{\cal S}_m}$, $m \ge 3$, is
called an $m\times m$-\emph{labyrinth pattern} (in short, \emph{labyrinth
pattern}) if  ${\cal A}$ satisfies 
\begin{itemize}
\item 
\textbf{The tree property.}
$\mathcal{G}({\cal A})$ is a tree.
\item 
\textbf{The exits property.}
${\cal A}$ has exactly one vertical exit pair, and exactly one horizontal exit pair.
\item \textbf{The corner property.}
If there is a white square in ${\cal A}$ at a corner of ${\cal A}$, then there is no white square in ${\cal A}$ at the diagonally opposite corner of ${\cal A}$. 
\end{itemize}
\end{definition}
\begin{definition}
For any labyrinth pattern ${\cal A}\in {\cal S}_m$ whose horizontal and vertical exit pair lies in row $r$ and column $c$, respectively, and $r, c\in \{1,\dots,m  \}$, we call the ordered pair $(r,c)$ the \emph{exits positions pair} of the pattern ${\cal A}$.
\end{definition}

\begin{assumptions}\label{ass:consistency} Let $(\widetilde{\mathcal A}_k)_{k\ge 1}$ be a sequence of collections of labyrinth patterns, where 
$\widetilde{\mathcal A}_k=\{{\mathcal A}_{k,h}, h=1,\dots, s_k\}$, with
$s_1=1$, and $s_k\ge 1$ for $k\ge 2$.  
\begin{itemize} 
\item \textbf{Pairwise tree consistency.}
For all $n\ge 1$, $\phi_{n+1}$ has the following property: if $W,W' \in {\mathcal W}_n$ are neighbours in ${\mathcal G}({\mathcal W}_n)$, then 
the restriction of the graph ${\cal G}({\mathcal W}_{n+1})$ to the subset of vertices that correspond to the (white) squares of level $n$ that are contained in $W'$ and $W''$  is a tree, for all neighbouring squares $W,W' \in {\mathcal W}_n$. 
\item \textbf{Exits consistency.}
All patterns in the collection $\widetilde{\mathcal A}_k$ have the same exits positions pair $(r_k,c_k)$ and the same width $m_k$, for all $k\ge 1$.
\item \textbf{Corner consistency.} For all $k\ge 1$, if in a pattern $A \in{\widetilde{\mathcal A}_k}$ there is a white square in a corner then there exists no pattern $A' \in{\widetilde{\mathcal A}_k}$ with a white square at the diagonally opposite corner. 
\end{itemize}
\end{assumptions}

Throughout this article we assume that the Assumptions \ref{ass:consistency} hold. 

\begin{proposition}\label{proposition:prop123}
${\mathcal W}_n$ has the tree property, the exits property and the corner property, i.e.,  ${\mathcal W}_n$ is a labyrinth pattern.
\end{proposition}

\begin{proof}
The proof works by induction. For ${\mathcal W}_1={\mathcal A}_{1,1}$ all three
properties are satisfied, since ${\mathcal A}_{1,1}$ is a labyrinth pattern.
Let $n\ge 2.$ We assume that ${\mathcal W}_{n-1}$ has the tree property,
the exits property and the corner property, and show that then ${\mathcal W}_{n}$
also has these three properties. The corner property follows immediately
from the corner property of ${\mathcal W}_{n-1}$ and from the corner property
and the corner hypothesis satified by the labyrinth patterns in
$\widetilde{\mathcal A}_n$.

The exits property of  ${\mathcal W}_n$ follows immediately from the exits 
property of ${\mathcal W}_{n-1}$ and from the exits property and the
exits hypothesis for the labyrinth patterns in $\widetilde{\mathcal A}_n$.

In order to prove that  ${\mathcal W}_n$  has the tree property, we proceed analogously as in the case of mixed or self-similar labyrinth fractals. ${\mathcal G}({\mathcal W}_n)$ is connected, by the connectedness of  the tree  ${\mathcal G}({\mathcal W}_{n-1})$, and by the exits property and the exits hypothesis for the patterns in $\widetilde{\mathcal A}_n$. Now, we give an indirect proof for the fact that ${\mathcal G}({\mathcal W}_n)$ has no cycles. Therefore, we assume that there is a cycle $C=\{u_0,u_1,\dots,u_r\}$ of minimal length in ${\mathcal G}({\mathcal W}_n).$ For $u \in {\mathcal V}({\mathcal G}({\mathcal W}_n))$ let $w(u)$ be the white square in ${\mathcal V}({\mathcal G}({\mathcal W}_{n-1}))$ which contains $u$ as a subset. Let $j_0=0$, $v_0=w(u_0)$, $j_k=\min\{i~:~w(u_i)\ne v_{k-1}, ~j_{k-1}<i \le r\}$, and $v_k=w(u_{j_k}),$ for $k \ge 1.$

Let $m$ be minimal such that the set $\{i~:~w(u_i)\ne v_m, ~j_m<i\le r\}$ is empty. Then we have, in ${\mathcal G}({\mathcal W}_{n-1}),$ $v_{i-1}\sim v_i,$ for $1 \le i \le m.$ 

If  ${\mathcal G}({\mathcal W}_{n-1})$ induced on the set $\{v_0, v_1, \dots,
v_m \}$ contains a cycle in  ${\mathcal G}({\mathcal W}_{n-1}),$ this
contradicts the induction hypothesis. Since by the tree property of labyrinth
patterns  ${\mathcal G}({\mathcal A}_{n,j})$ is a tree, for all ${\mathcal
A}_{n,j}\in \widetilde{\mathcal A}_{n},$ it follows that not all white squares
of the cycle $C$ in ${\mathcal G}({\mathcal W}_{n})$ can be contained in $v_0$,
wherefrom it follows that $m\ge 1.$ Thus  ${\mathcal G}({\mathcal W}_{n})$
induced on the set $\{v_0, v_1, \dots, v_m \}$ is a tree with more than one
vertex. From the cycle-free mixing hypothesis we obtain the case $m=2$ is not
possible, thus $m\ge 3$.  The fact that  ${\mathcal G}({\mathcal W}_{n})$
induced on the set $\{v_0, v_1, \dots, v_m \}$ is a tree implies that, in oder
for the cycle $C$ to exist in ${\mathcal G}({\mathcal W}_n)$ there exists a
square $v_k \in {\mathcal W}_{n-1}),$ in the mentioned tree, such that in this
square is ``crossed'' by two distinct paths in ${\mathcal G}({\mathcal
W}_{n})$, i.e., there exist two disjoint paths $p_1,p_2$ in ${\mathcal
G}({\mathcal W}_{n})$ that connect pairs of white squares of level $n$ that lie
on the same two (distinct) sides of $v_k$, e.g., the top and bottom side, or
the top and left side or, due to symmetry arguments, one can chose to any other
such pair of sides. Since $\tilde {\mathcal A}_n$ consists only of labyrinth
patterns, the graph ${\mathcal G}({\mathcal W}_{n})$ restricted to the squares
that lie inside $v_k$ is connected, and thus there has to exist a path in
${\mathcal G}({\mathcal W}_{n})$, inside $v_k$,  that connects a square of
level $n$ lying on $p_1$ with a square of level $n$ that lies on $p_2$. This
produces a new cycle, shorter than $C$.

This contradicts the assumption that $u_0, u_1,\dots, u_r$ is a cycle of minimal length.

\end{proof}

\begin{remark} The tree property implies the uniqueness of a cycle-free path between any distinct (white) squares in ${\mathcal V}({\mathcal G}({\mathcal A}))$, for any labyrinth pattern ${\mathcal A}$,  and any distinct (white) squares in ${\mathcal V}({\mathcal G}({\mathcal W}_n))$, for any supermixed labyrinth set ${\mathcal W}_n$ of level $n$.
\end{remark}

\begin{definition}
For $n\ge 2$ we call ${\mathcal W}_n$ an $m(n)\!\!\times\!\!{}m(n)$ supermixed labyrinth set (in short, supermixed labyrinth set) of level $n$, and the limit set 
$$ L_{\infty}=\bigcap _{n\ge 1}\bigcup_{W \in {\mathcal W}_n} W$$
the supermixed labyrinth fractal generated by the sequence of collections of labyrinth patterns $(\widetilde{\mathcal A}_k)_{k\ge 1}$, 
\end{definition}

\begin{remark} One can immediately see that supermixed labyrinth fractals generalise mixed labyrinth fractals and self-similar labyrinth fractals. 
If $s_k=1$, for all $k\ge 1$, then ${\mathcal W}_n$ is a mixed labyrinth set of level $n$, and $L_{\infty}$ a mixed labyrinth fractal, as defined in \cite{mixlaby}. If we use only one pattern throughout the construction, we recover the
self-similar case from \cite{laby_4x4,laby_oigemoan}.
\end{remark}

\section{Topological properties of supermixed labyrinth fractals}
\label{sec:topological}

Recall that for any sequence $\{\widetilde{\cal
A}_k\}_{k=1}^{\infty}$ of collections of labyrinth patterns 
we assert Assumptions \ref{ass:consistency} from Section \ref{sec:notions}.

\begin{lemma}\label{lemma:Steinhaus} Let $\{\widetilde{\cal
A}_k\}_{k=1}^{\infty}$ be a sequence of collections of labyrinth patterns
 and let $n\ge 1$. Then, from every black
square in $\mathcal{G}({\cal B}_{n})$ there is a path in $\mathcal{G}({\cal
B}_{n})$ to a black square of level $n$ in $\mathcal{G}({\cal B}_{n})$.
\end{lemma}
\begin{proof} The proof  uses Proposition \ref{proposition:prop123} and works in the same way as in the case of labyrinth sets occuring in the construction of self-similar 
labyrinth fractals \cite[Lemma 2]{laby_4x4}.
\end{proof}

We state the following two results without the proof, since the proofs given in the self-similar case \cite{laby_4x4} work also in the more general case of the 
supermixed labyrinth sets. For more details and definitions we refer to the papers \cite{laby_4x4, laby_oigemoan}.

\begin{lemma}\label{lemma:border} Let $\{\widetilde {\cal A}_k\}_{k=1}^{\infty}$ be a sequence of collections of labyrinth patterns. If $x$ is a point in $([0,1]\times [0,1]) \setminus L_n$, then there is an arc $a \subseteq ([0,1]\times [0,1]) \setminus L_{n+1}$ between $x$ and a point in the boundary $\fr([0,1]\times [0,1])$.
\end{lemma}
\begin{corollary}\label{corollary:border} Let $\{\widetilde{\cal A}_k\}_{k=1}^{\infty}$ be a sequence of collections of labyrinth patterns and let $n\ge 1$. If $x$ is a point in $([0,1]\times [0,1]) \setminus L_{\infty}$, then there is an arc $a \subseteq ([0,1]\times [0,1]) \setminus L_{\infty}$ between $x$ and a point in $\fr([0,1]\times [0,1])$.
\end{corollary}

Now, let us recall that a \emph{continuum} is a compact connected Hausdorff space, and a \emph{dendrite} is a locally connected continuum that contains no simple closed curve.

\begin{theorem}\label{theorem:dendrite}
Any supermixed labyrinth fractal $L_{\infty}$ generated by a sequence of collections of labyrinth patterns is a dendrite.
\end{theorem}

\begin{proof}
The proof, which uses the 
Hahn-Mazurkiewicz-Sierpi\'nski theorem \cite[Theorem~2, p.256]{Kuratowski} 
as well as the Jordan Curve theorem, is almost identical to the the proof
of Theorem 1 in \cite{mixlaby} but uses 
Corollary~\ref{corollary:border} in place of the corresponding result there.
\end{proof}

\noindent As a consequence of the fact that $L_{\infty}$ is a dendrite, for any two  points  $x\ne y$ in $L_{\infty}$ there exists a unique arc in
$L_{\infty}$ that connects them  \cite[Theorem~3, par. 47, V, p. 181]{Kuratowski}. Throughout this paper we denote by ${{a}}(x,y)$ the unique arc in $L_{\infty}$ with endpoints $x$ and $y$.

\section{Exits}\label{sec:exits}
Let $W_n^{\etop}\in {\cal W}_{n}$ be the top exit of ${\cal W}_{n}$, for $n\ge 1$. We call $\bigcap_{n=1}^{\infty}W_n^{\etop}$ the \emph{top exit of} 
$L_{\infty}$. The other exits of $L_{\infty}$, $W_n^{\ebottom}, W_n^{\eleft}, W_n^{\eright},$ are defined analogously. We note that 
the exits property for ${\cal W}_n$ yields that $(x,1),(x,0)\in L_{\infty}$ if and only if $(x,1)$ is the top exit of $L_{\infty}$ and $(x,0)$ 
is the bottom exit of $L_{\infty}$. The analogous holds also for the left and the right exit of the fractal. 

Let $n\ge 1$, $W\in {\cal W}_{n}$, and $t$ be the intersection of $L_{\infty}$ with the top edge of $W$. 
Then we call $t(W)$ the \emph{top exit} of $W$. Analogously we define the \emph{bottom exit} $b(W)$, the\emph{ left exit} $l(W)$ and
the \emph{right exit} $r(W)$ of $W$.
 Each of these four exits is unique, by the uniqueness of the four exits of a 
supermixed mixed labyrinth fractal and by the fact that each 
set of the form 
$L_{\infty} \cap W$, with $W\in {\cal W}_{n}$, is a supermixed labyrinth fractal scaled by the factor $m(n)$.
Thus, we have defined the notion of \emph{exit} for 
three different types of objects: for supermixed labyrinth sets of level $n$, for $L_{\infty}$, 
and for squares in ${\cal W}_{n}$, for $n\ge 1$.

The following result immediately follows from the construction of supermixed labyrinth fractals.
\begin{proposition} \label{prop:exits_coordinates}Let $\{\widetilde{\cal A}_k\}_{k=1}^{\infty}$ be a sequence of collections of 
labyrinth patterns, as above.
 \begin{itemize}
 \item[(a)] 
If $(x_1^t, x_2^t)$ and $(x_1^b, x_2^b)$ are the Cartesian coordinates of the top exit and bottom exit, respectively, 
in $L_{\infty}$,
then
\[
x_1^t=x_1^b=\sum_{k=1}^{\infty}\frac{c_k-1}{m(k)}, ~~x_2^t=1, ~~x_2^b=0.
 \]
 \item[(b)] 
If ${x_1^{l}}, x_2^{l}$ and $x_1^r, x_2^r$ are the Cartesian coordinates of the left exit and right exit, respectively, in $L_{\infty}$,
then
\[
x_2^{l}=x_2^r=\sum_{k=1}^{\infty}\frac{m_k-r_k}{ m(k)}, ~~x_1^{l}=0, ~~x_1^r=1.
 \]
 \end{itemize}
\end{proposition}

\section{Paths in supermixed labyrinth sets. Path matrices and counting matrices} \label{sec:path_matrices}

As in the setting of self-similar labyrinth fractals or mixed labyrinth fractals, the first step is to study paths in prefractals, i.e., in this case in (the graphs of) supermixed labyrinth sets. More precisely, we look at the construction method for paths between distinct exits of supermixed labyrinth sets. 
Sometimes we will just skip ``supermixed'' when it is understood from the context that we deal with supermixed objects.

Therefor, let us first introduce some notation.
We call a path in $\mathcal{G}({\cal W}_{n})$ a $\A$\emph{-path} if it leads from the top to 
the bottom exit of ${\cal W}_n$. 
The $\B,\C,\D,\E$, and $\F$\emph{-paths} lead from left to right, top to right, right to bottom, bottom to left, and left to 
top exits, respectively. Formally, we denote such a path by $\p_i ({\cal W}_n),$  with $i\in \{\A,\B,\C,\D,\E,\F\}.$  Let $\A(n),\B(n),\C(n),\D(n),\E(n)$, and $\F(n)$ be the 
length of the respective path in 
$\mathcal{G}({\cal W}_{n})$, for $n\ge 1,$ and 
 $\A_{k,h},\B_{k,h},\C_{k,h},\D_{k,h},\E_{k,h}$, and $\F_{k,h}$ the length of the respective path in $\mathcal{G}({\cal A}_{k,h})$, 
for $k\ge 1$ and $1\le h \le s_k$. Formally, we denote such a path in  (the graph of) this pattern by $\p_i({\cal A}_{k,h}),$ with $i\in \{\A,\B,\C,\D,\E,\F\}.$ By the length of a path in a labyrinth set (of level $n$) or labyrinth pattern we mean the number of squares (of level $n$) in the path.
For $n=k=1$ the two path lengths coincide, i.e., $\A(1)=\A_{1,1}, \dots, \F(1)=\F_{1,1}$.

Both in the case of self-similar fractals \cite{laby_4x4, laby_oigemoan} and in the case of mixed labyrinth fractals \cite{mixlaby} we used the same method for the construction of paths in prefractals. Before describing this construction method for the case of supermixed labyrinth sets, let us remark some facts that play an essential role in the reasoning about paths in supermixed labyrinth sets and arcs in supermixed labyrinth fractals.

\vspace{0.2cm}
\begin{remark} One can prove (e.g., by using an indirect proof) that if a square $W\in {\mathcal V}({\mathcal G}({\mathcal W}_n))$ is one of the four exits of ${\mathcal W}_n$, then it has as a subset the square $U\in {\mathcal V}({\mathcal G}({\mathcal W}_{n+1}))$ which is the exit in ${\mathcal W}_{n+1}$ of the same type (\emph{top, bottom, left}, or \emph{right}) as $W$ in ${\mathcal W}_{n}$ . Moreover, from the tree property of ${\mathcal G}({\mathcal W}_{n})$, for all $n\ge 1$, it also follows that for any distinct exits $W',W''$ of ${\mathcal W}_n$  the corridor $\Gamma(p_n(W',W''))=\bigcup_{W\in p_n(W',W'')} W$ of the path $p_n(W',W'')=\{W'=W_1, W_2, \dots, W_r=W'' \}$  in ${\mathcal G}({\mathcal W}_{n})$ contains as a subset the corridor $ \Gamma(p_{n+1}(U',U''))=\bigcup_{U\in p_{n+1}(U',U'')} U$ of the path $p_{n+1}(U',U'')=\{U'=U_1, U_2, \dots, U_r=U'' \}$ in ${\mathcal W}_{n+1}$, where $U'$ and $U''$ is, respectively, the same type of exit in ${\mathcal W}_{n+1}$,  as $W'$ and $W''$ in ${\mathcal W}_{n}$. We remark that the purpose of the index $n$ in the notation $p_n(W',W'')$ for the path is to emphasise that the path is in the graph ${\mathcal G}({\mathcal W}_n)$.
\end{remark}

\begin{figure}[hhhh]
\begin{center}
\includegraphics[scale=1]{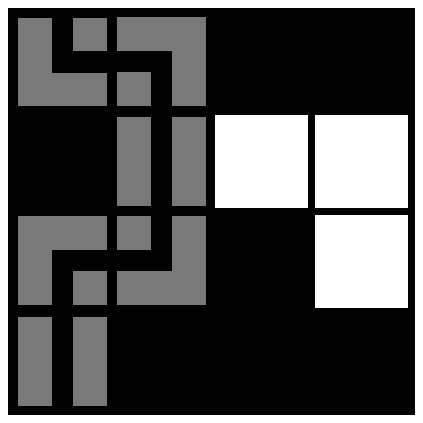}~~
\includegraphics[scale=1]{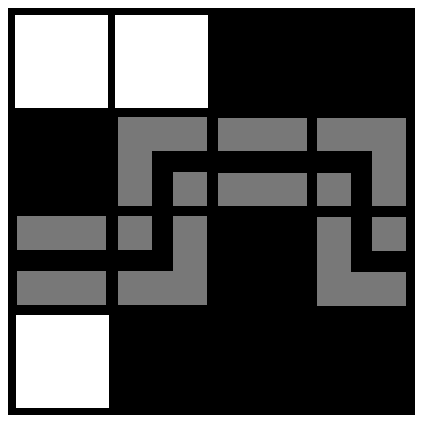}
\caption{Paths from top to bottom and from left to right exit of ${\cal A}_{1}$}\label{fig:Squares of type $A$ and $B$.}
\end{center}
\end{figure}
\begin{figure}[hhhh]
\begin{center}
\includegraphics[scale=1]{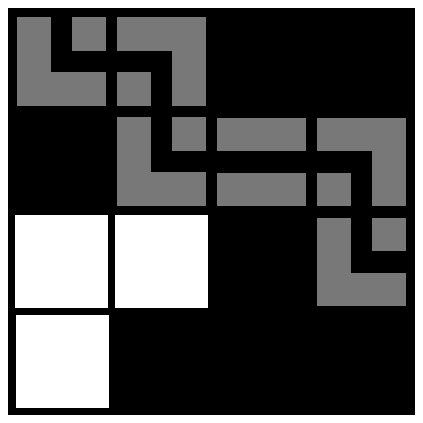}~~
\includegraphics[scale=1]{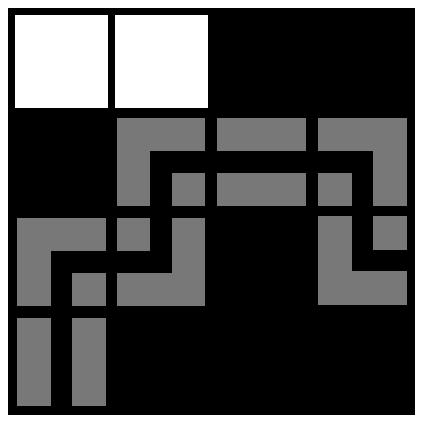}
\caption{Paths from top to right and from bottom to right exit of ${\cal A}_{1}$}\label{fig:Squares of type $C$ and $D$.}
\end{center}
\end{figure}
\begin{figure}[hhhh]
\begin{center}
\includegraphics[scale=1]{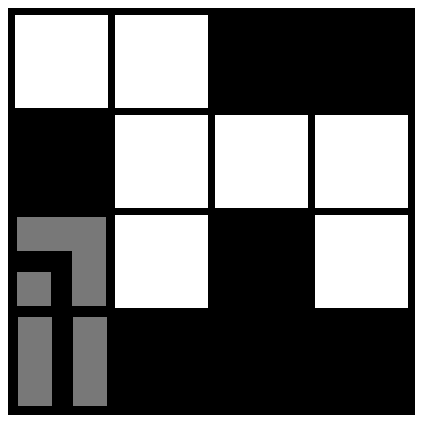}~~
\includegraphics[scale=1]{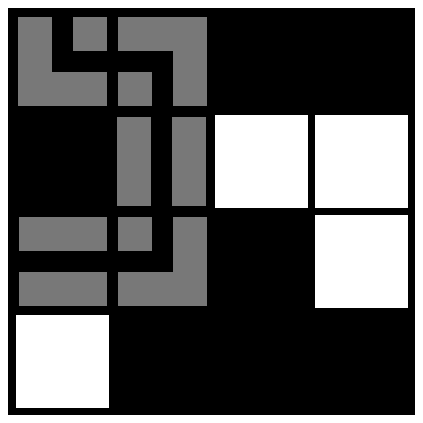}
\caption{Paths from left to bottom and top to left exit of ${\cal A}_{1}$}\label{fig:Squares of type $E$ and $F$.}
\end{center}
\end{figure}

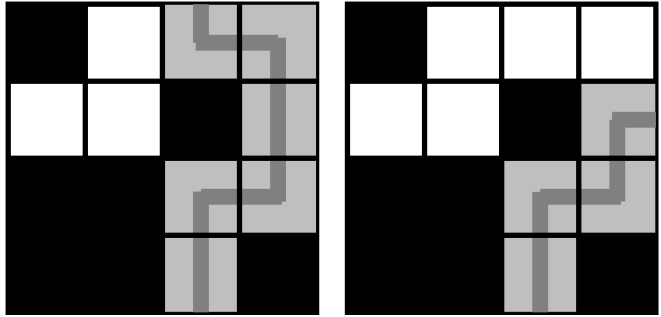
\begin{figure}[h!]
\begin{center}
\begin{tikzpicture}[scale=.41]
\draw[line width=2pt] (0,0) rectangle (10,10);

\filldraw[fill=gray!50, draw= black] (5,0) rectangle (7.5,5);
\filldraw[fill=gray!50, draw= black] (5,7.5) rectangle (7.5,10);
\filldraw[fill=gray!50, draw= black] (7.5,2.5) rectangle (10,10);
\draw[line width=6 pt, color=gray] (6.25, 0) -- (6.25, 3.92);
\draw[line width=6 pt, color=gray] (6.25,3.75) -- (8.87,3.75);
\draw[line width=6 pt, color=gray] (8.75, 3.59) -- (8.75, 8.91);
\draw[line width=6 pt, color=gray] (8.75, 8.75) -- (6.08, 8.75);
\draw[line width=6 pt, color=gray] (6.25, 8.75) -- (6.25, 10);
\draw[line width=2pt] (2.5, 0) -- (2.5,10);
\draw[line width=2pt] (5, 0) -- (5,10);
\draw[line width=2pt] (7.5, 0) -- (7.5,10);
\draw[line width=2pt] (0, 2.5) -- (10,2.5);
\draw[line width=2pt] (0, 5) -- (10,5);
\draw[line width=2pt] (0, 7.5) -- (10,7.5);
\filldraw[fill=black, draw=black] (0,0) rectangle (5, 5);
\filldraw[fill=black, draw=black] (0,7.5) rectangle (2.5, 10);
\filldraw[fill=black, draw=black] (5,5) rectangle (7.5, 7.5);
\filldraw[fill=black, draw=black] (7.5,0) rectangle (10, 2.5);

\draw[line width=2pt] (11,0) rectangle (21,10);

\filldraw[fill=gray!50, draw= black] (16,0) rectangle (18.5,5);
\filldraw[fill=gray!50, draw= black] (18.5,2.5) rectangle (21,7.5);

\draw[line width=6 pt, color=gray] (17.25, 0) -- (17.25, 3.91);
\draw[line width=6 pt, color=gray] (17.25,3.75) -- (19.87,3.75);
\draw[line width=6 pt, color=gray] (19.75, 3.60) -- (19.75, 6.25);
\draw[line width=6 pt, color=gray] (19.58, 6.25) -- (21, 6.25);

\draw[line width=2pt] (13.5, 0) -- (13.5,10);
\draw[line width=2pt] (16, 0) -- (16,10);
\draw[line width=2pt] (18.5, 0) -- (18.5,10);
\draw[line width=2pt] (11, 2.5) -- (21,2.5);
\draw[line width=2pt] (11, 5) -- (21,5);
\draw[line width=2pt] (11, 7.5) -- (21,7.5);
\filldraw[fill=black, draw=black] (11,0) rectangle (16, 5);
\filldraw[fill=black, draw=black] (11,7.5) rectangle (13.5, 10);
\filldraw[fill=black, draw=black] (16,5) rectangle (18.5, 7.5);
\filldraw[fill=black, draw=black] (18.5,0) rectangle (21, 2.5);

\end{tikzpicture}
\caption{Paths from bottom to top and from bottom to right exit of $\mathcal{A}_{2,1}$}
\label{fig:A21_paths}
\end{center}
\end{figure}

\begin{figure}[h!]
\begin{center}
\begin{tikzpicture}[scale=.41]
\draw[line width=2pt] (0,0) rectangle (10,10);
\draw[line width=2pt] (11,0) rectangle (21,10);

\filldraw[fill=gray!50, draw= black] (0,5) rectangle (5,7.5);
\filldraw[fill=gray!50, draw= black] (2.5,2.5) rectangle (10,5);
\filldraw[fill=gray!50, draw= black] (7.5,5) rectangle (10,7.5);
\draw[line width=6 pt, color=gray] (0, 6.25) -- (3.99, 6.25);
\draw[line width=6 pt, color=gray] (8.75, 3.59) -- (8.75, 6.25);
\draw[line width=6 pt, color=gray] (3.75,6.30 ) -- (3.75, 3.75);
\draw[line width=6 pt, color=gray] (3.51,3.75 ) -- (8.92, 3.75);
\draw[line width=6 pt, color=gray] (8.50 ,6.25 ) -- (10, 6.25);

\draw[line width=2pt] (2.5, 0) -- (2.5,10);
\draw[line width=2pt] (5, 0) -- (5,10);
\draw[line width=2pt] (7.5, 0) -- (7.5,10);
\draw[line width=2pt] (0, 2.5) -- (10,2.5);
\draw[line width=2pt] (0, 5) -- (10,5);
\draw[line width=2pt] (0, 7.5) -- (10,7.5);
\filldraw[fill=black, draw=black] (0,0) rectangle (2.5, 5);
\filldraw[fill=black, draw=black] (0,7.5) rectangle (2.5, 10);
\filldraw[fill=black, draw=black] (5,5) rectangle (7.5, 7.5);
\filldraw[fill=black, draw=black] (2.5,0) rectangle (5, 2.5);
\filldraw[fill=black, draw=black] (7.5,0) rectangle (10, 2.5);
\filldraw[fill=black, draw=black] (7.5,7.5) rectangle (10, 10);

\filldraw[fill=gray!50, draw= black] (11,5) rectangle (16,7.5);
\filldraw[fill=gray!50, draw= black] (13.5,7.5) rectangle (18.5,10);

\draw[line width=6 pt, color=gray] (11, 6.25) -- (15.08, 6.25);
\draw[line width=6 pt, color=gray] (14.92 ,6.10 ) -- ( 14.92, 8.98);
\draw[line width=6 pt, color=gray] ( 14.92, 8.75)  -- ( 17.45 , 8.75);
\draw[line width=6 pt, color=gray] ( 17.25 , 8.75) -- ( 17.25, 10 );

\draw[line width=2pt] (13.5, 0) -- (13.5,10);
\draw[line width=2pt] (16, 0) -- (16,10);
\draw[line width=2pt] (18.5, 0) -- (18.5,10);
\draw[line width=2pt] (11, 2.5) -- (21,2.5);
\draw[line width=2pt] (11, 5) -- (21,5);
\draw[line width=2pt] (11, 7.5) -- (21,7.5);
\filldraw[fill=black, draw=black] (11,0) rectangle (13.5, 5);
\filldraw[fill=black, draw=black] (13.5,0) rectangle (16, 2.5);
\filldraw[fill=black, draw=black] (11,7.5) rectangle (13.5, 10);
\filldraw[fill=black, draw=black] (16,5) rectangle (18.5, 7.5);
\filldraw[fill=black, draw=black] (18.5,0) rectangle (21, 2.5);
\filldraw[fill=black, draw=black] (18.5,7.5) rectangle (21, 10);

\end{tikzpicture}
\caption{Paths from left to right and from left to top exit of $\mathcal{A}_{2,2}$}
\label{fig:A22_paths}
\end{center}
\end{figure}
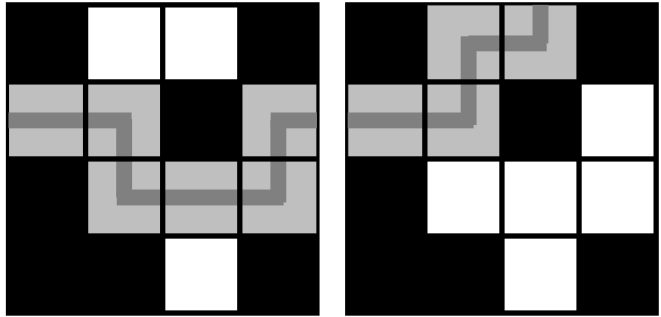

\begin{figure}[h!]
\begin{center}
\begin{tikzpicture}[scale=.3]
\draw[line width=1pt] (0,0) rectangle (16,16);
\filldraw[fill=gray!95, draw= black] (0,0) rectangle (4,4);
\filldraw[fill=gray!95, draw= black] (0,4) rectangle (4,8);
\filldraw[fill=gray!95, draw= black] (4,4) rectangle (8,8);
\filldraw[fill=gray!95, draw= black] (4,8) rectangle (8,12);
\filldraw[fill=gray!95, draw= black] (8,8) rectangle (12,12);
\filldraw[fill=gray!95, draw= black] (12,8) rectangle (16,12);
\filldraw[fill=gray!95, draw= black] (12,4) rectangle (16,8);
\filldraw[fill=gray!50, draw= black] (2,0) rectangle (3,2);
\filldraw[fill=gray!50, draw= black] (3,1) rectangle (4,4);
\filldraw[fill=gray!50, draw= black] (2,3) rectangle (3,6);
\filldraw[fill=gray!50, draw= black] (3,5) rectangle (4,7);
\filldraw[fill=gray!50, draw= black] (4,6) rectangle (5,7);
\filldraw[fill=gray!50, draw= black] (5,6) rectangle (6,8);
\filldraw[fill=gray!50, draw= black] (6,7) rectangle (7,10);
\filldraw[fill=gray!50, draw= black] (7,9) rectangle (8,11);
\filldraw[fill=gray!50, draw= black] (8,10) rectangle (10,11);
\filldraw[fill=gray!50, draw= black] (9,9) rectangle (12,10);
\filldraw[fill=gray!50, draw= black] (11,10) rectangle (14,11);
\filldraw[fill=gray!50, draw= black] (13,9) rectangle (15,10);
\filldraw[fill=gray!50, draw= black] (14,7) rectangle (15,9);
\filldraw[fill=gray!50, draw= black] (15,6) rectangle (16,8);

\draw[line width=0.8pt] (4, 0) -- (4,16);
\draw[line width=0.8pt] (8, 0) -- (8,16);
\draw[line width=0.8pt] (12, 0) -- (12,16);
\draw[line width=0.8pt] (0, 4) -- (16, 4);
\draw[line width=0.8pt] (0, 8) -- (16, 8);
\draw[line width=0.8pt] (0, 12) -- (16,12);
\draw[line width=0.5pt] (1, 0) -- (1,16);
\draw[line width=0.5pt] (2, 0) -- (2,16);
\draw[line width=0.5pt] (3, 0) -- (3,16);
\draw[line width=0.5pt] (5, 0) -- (5,16);
\draw[line width=0.5pt] (6, 0) -- (6,16);
\draw[line width=0.5pt] (7, 0) -- (7,16);
\draw[line width=0.5pt] (9, 0) -- (9,16);
\draw[line width=0.5pt] (10, 0) -- (10,16);
\draw[line width=0.5pt] (11, 0) -- (11,16);
\draw[line width=0.5pt] (13, 0) -- (13,16);
\draw[line width=0.5pt] (14, 0) -- (14,16);
\draw[line width=0.5pt] (15, 0) -- (15,16);
\draw[line width=0.5pt] (0, 1) -- (16,1);
\draw[line width=0.5pt] (0, 2) -- (16,2);
\draw[line width=0.5pt] (0, 3) -- (16,3);
\draw[line width=0.5pt] (0, 5) -- (16,5);
\draw[line width=0.5pt] (0, 6) -- (16,6);
\draw[line width=0.5pt] (0, 7) -- (16,7);
\draw[line width=0.5pt] (0, 9) -- (16,9);
\draw[line width=0.5pt] (0, 10) -- (16,10);
\draw[line width=0.5pt] (0, 11) -- (16,11);
\draw[line width=0.5pt] (0, 13) -- (16,13);
\draw[line width=0.5pt] (0, 14) -- (16,14);
\draw[line width=0.5pt] (0, 15) -- (16,15);
\filldraw[fill=black, draw=black] (0,8) rectangle (4, 12);
\filldraw[fill=black, draw=black] (4,0) rectangle (16, 4);
\filldraw[fill=black, draw=black] (8,4) rectangle (12, 8);
\filldraw[fill=black, draw=black] (8,12) rectangle (16, 16);
\filldraw[fill=black, draw=black] (0,0) rectangle (2, 2);
\filldraw[fill=black, draw=black] (0,3) rectangle (1, 4);
\filldraw[fill=black, draw=black] (2,2) rectangle (3, 3);
\filldraw[fill=black, draw=black] (3,0) rectangle (4, 1);

\filldraw[fill=black, draw=black] (0,4) rectangle (2, 6);
\filldraw[fill=black, draw=black] (0,7) rectangle (1, 8);
\filldraw[fill=black, draw=black] (2,6) rectangle (3, 7);
\filldraw[fill=black, draw=black] (3,4) rectangle (4, 5);
\filldraw[fill=black, draw=black] (4,8) rectangle (6, 10);
\filldraw[fill=black, draw=black] (4,11) rectangle (5, 12);
\filldraw[fill=black, draw=black] (6,10) rectangle (7, 11);
\filldraw[fill=black, draw=black] (7,8) rectangle (8, 9);
\filldraw[fill=black, draw=black] (4,12) rectangle (6, 14);
\filldraw[fill=black, draw=black] (4,15) rectangle (5, 16);
\filldraw[fill=black, draw=black] (6,14) rectangle (7, 15);
\filldraw[fill=black, draw=black] (7,12) rectangle (8, 13);
\filldraw[fill=black, draw=black] (12,4) rectangle (14, 6);
\filldraw[fill=black, draw=black] (12,7) rectangle (13, 8);
\filldraw[fill=black, draw=black] (14,6) rectangle (15, 7);
\filldraw[fill=black, draw=black] (15,4) rectangle (16, 5);
\filldraw[fill=black, draw=black] (0,12) rectangle (1, 14);
\filldraw[fill=black, draw=black] (0,15) rectangle (1, 16);
\filldraw[fill=black, draw=black] (1,12) rectangle (2, 13);
\filldraw[fill=black, draw=black] (2,14) rectangle (3, 15);
\filldraw[fill=black, draw=black] (3,12) rectangle (4, 13);
\filldraw[fill=black, draw=black] (3,15) rectangle (4, 16);
\filldraw[fill=black, draw=black] (4,4) rectangle (5, 6);
\filldraw[fill=black, draw=black] (4,7) rectangle (5,
 8);
\filldraw[fill=black, draw=black] (5,4) rectangle (6, 5);
\filldraw[fill=black, draw=black] (6,6) rectangle (7, 7);
\filldraw[fill=black, draw=black] (7,4) rectangle (8, 5);
\filldraw[fill=black, draw=black] (7,7) rectangle (8, 8);
\filldraw[fill=black, draw=black] (8,8) rectangle (9, 10);
\filldraw[fill=black, draw=black] (8,11) rectangle (9, 12);
\filldraw[fill=black, draw=black] (9,8) rectangle (10, 9);
\filldraw[fill=black, draw=black] (10,10) rectangle (11, 11);
\filldraw[fill=black, draw=black] (11,8) rectangle (12, 9);
\filldraw[fill=black, draw=black] (11,11) rectangle (12, 12);
\filldraw[fill=black, draw=black] (12,8) rectangle (13, 10);
\filldraw[fill=black, draw=black] (12,11) rectangle (13, 12);
\filldraw[fill=black, draw=black] (13, 8) rectangle (14, 9);
\filldraw[fill=black, draw=black] (14,10) rectangle (15, 11);
\filldraw[fill=black, draw=black] (15, 8) rectangle (16, 9);
\filldraw[fill=black, draw=black] (15,11) rectangle (16, 12);

\end{tikzpicture}
\caption{The path (in lighter grey) from the bottom 
to the right exit of the labyrinth set $ \mathcal{W}_2$ shown in Figure \ref{fig:W2} } 
\label{fig:W2_path}
\end{center}
\end{figure}
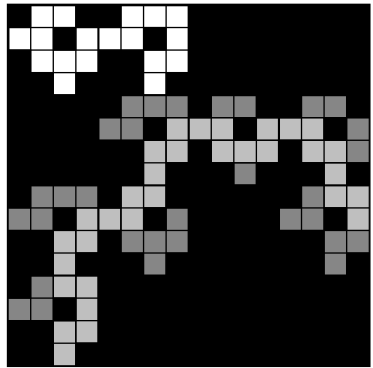

In the following, let us describe the construction of paths that connect exits
of supermixed labyrinth sets of level  $n\ge 1.$ Therefor, let us consider,
e.g., the path between the bottom and the right exit in $W_n$, for some fixed
$n \ge 1.$ For any other pair of distinct exits in $W_n$, the construction
follows the same steps.  Let us assume $n\ge 2,$ since for $n = 1$ the
mentioned path concides with the path from the bottom to the right exit of the
pattern ${\mathcal A}_{1,1}$ and we are done. 

The first step is to find the path between the bottom and the right exit of ${\mathcal W}_1,$ which is equivalent to constructing the path between the bottom and the right exit of ${\mathcal A}_{1,1}$. Now we denote each square in this path according to its neighbours within this path, i.e., we establish its type: if it has a top and a bottom neighbour it is called a $\A$-\emph{square}  
(with respect to the path), and
it is called a $\B,\C,\D,\E$, and $\F$-\emph{square} if its neighbours are at left-right,
top-right, 
bottom-right, left-bottom, and left-top, respectively. 
If a square in the mentioned path is an exit of the pattern (the labyrinth set), it is supposed to have a neighbour outside the 
side of the exit. A bottom exit, e.g., is supposed to have a neighbour below, outside the bottom, additionally to its inside neighbour.
   We repeat this procedure for all possible paths between two exits in ${\cal G}({\cal W}_{1})$, as shown in 
 Figure \ref{fig:Squares of type $A$ and $B$.}, \ref{fig:Squares of type $C$ and $D$.}, and \ref{fig:Squares of type $E$ and $F$.}.

We introduce the notation ${\cal J}=\{\A,\B,\C,\D,\E,\F \}$.

Now, as a next step, in order to construct the $\D$-path in $\mathcal{G}({\cal
W}_{2})$, which is shown in Figure~\ref{fig:W2_path}, we replace each
$j$-square $W$ of the $\D$-path in ${\mathcal G}({\cal W}_{1})$ with the
$j$-path in $\mathcal{G}(\phi_2(W))$, where $j\in  {\cal J}$. 
Some of the paths in the patterns in $\widetilde {\mathcal A}_2$ are shown in 
Figures~\ref{fig:A21_paths} and \ref{fig:A22_paths}.  
 In general, for any pair of exits and $n\ge 1$, we replace each marked
square $W$ in the path in $\mathcal{G}({\cal W}_{n})$ with its corresponding path in the graph of the pattern $\phi_{n+1}(W)$ and thus obtain the path of the supermixed labyrinth set
$\mathcal{G}({\cal W}_{n+1})$.

 \begin{notation*} For any square $W$ in a given path between two exits of a labyrinth pattern or labyrinth set, if $W$ is a square of type $j \in {\cal J}$ within this path, we write $\ty (W)=j.$
\end{notation*}

\vspace{0.2cm}
We recall the definition of path matrices of labyrinth patterns or labyrinth sets of some level. 
\begin{definition}
The path matrix of  a labyrinth pattern ${\mathcal A}$ is 
a non-negative $6 \times 6$ 
matrix whose rows and columns are both indexed according to the set ${\cal J}=\{\A,\B,\C,\D,\E,\F \}$ and whose entry on the position $(i,j)\in {\cal J}\times {\cal J}$ is the number of squares of type $j$ in the path of type $i$ in  ${\mathcal G}({\mathcal A}).$ If ${\mathcal A}={\mathcal W}_n$, then we obtain the path matrix of the labyrinth set  ${\mathcal W}_n$ of level $n\ge 1$.
\end{definition}

For a sequence of collections of labyrinth patterns $\{\widetilde{\mathcal A}_k
\}_{k\ge 1}$, with $\widetilde{\mathcal A}_k=\{{\mathcal A}_{k,1},
\dots,{\mathcal A}_{k,s_k}\}$, for $k\ge 1$ and the corresponding sequence $\{
{\cal W}_n\}_{n\ge 1}$ of supermixed labyrinth sets of level $n$, we introduce
in the following some more notation for matrices that are useful in the study
of paths in supermixed labyrinth sets.  
\begin{notation*}
For all $n\ge 1$ we denote the \emph{path matrix of the labyrinth set} ${\cal W}_n$ by $$M({\cal W}_n):=M(n)=(m^{(n)}_{i,j})_{i,j\in {\cal J}}.$$ 

\noindent For all $k\ge 1$ and $h=1,\dots, s_k$ we denote the \emph{path matrix of the labyrinth pattern} ${\cal A}_{k,h}$ by $$M({\cal A}_{k,h}):=M_{k,h}=(m^{k,h}_{i,j})_{i,j\in {\cal J}}.$$ 
\end{notation*}

\begin{definition}
For $n \ge 1$ and $h\in \{1,\dots, s_{n+1}\}$ the $h$-th \emph{counting matrix for step} $n$, $Q_{n,h}=(q^{n,h}_{i,j})_{i,j \in {\cal J}}$, is defined by
\begin{equation}
\label{eq:defmatrixQ}
q^{n,h}_{i,j}:=\sum_{W \in \p_i({\mathcal W}_n)}\mathbf{1}_{\phi_{n+1}(W)={\mathcal A}_{n+1,h}}\cdot \mathbf{1}_{\ty(W)=j}. 
\end{equation}
In the above formula, $\mathbf{1}$ denotes in each case the corresponding indicator function. In other words, the entry $q^{n,h}_{i,j}$ is the number of $j$-squares in the path of type $i$ in ${\mathcal G}({\mathcal W}_n)$ which at the next step are ``substituted'' according to the pattern ${\cal A}_{n+1,h}$, with $i,j,n,h$ as above.
\end{definition}

The following proposition is an immediate consequence of the definitions of the path
matrices and the construction method given at the beginning of this section.

\begin{proposition}
With the above notation we have

 \begin{equation}\label{eq:matrix_k,h}
 \left(
 \begin{array}{l}
\A_{k,h} \\  
\B_{k,h} \\  
\C_{k,h} \\  
\D_{k,h} \\  
\E_{k,h} \\  
\F_{k,h} \\  
\end{array}
\right)
=M_{k,h} \cdot \left(
\begin{array}{l}
1 \\  
1 \\  
1 \\  
1 \\  
1 \\  
1 \\  
\end{array}\right)
\text{ and }
 \left(
 \begin{array}{l}
\A(n) \\  
\B(n) \\  
\C(n) \\  
\D(n) \\  
\E(n) \\  
\F(n) \\  
\end{array}
 \right)
=M(n) \cdot \left(
\begin{array}{l}
1 \\  
1 \\  
1 \\  
1 \\  
1 \\  
1 \\  
\end{array}\right).
\end{equation}
\end{proposition}

The next result establishes identities that describe the relations between path
matrices of labyrinth patterns, supermixed labyrinth sets of some level and the
counting matrices introduced before. In this theorem we extend 
results obtained for path matrices in the case of self-similar \cite{laby_4x4,
laby_oigemoan} and mixed labyrinth fractals \cite{mixlaby}.
\begin{theorem}
\label{theo:basisformel}
With the above notation we have, for all $n \ge 1,$ 
\begin{equation}
\label{eq:M(n)}
M(n)= Q_{n,1}+\dots+Q_{n,s_{n+1}}
\end{equation}
and
\begin{equation}
\label{eq:M(n+1)}
M(n+1)=\sum_{h=1}^{s_{n+1}} Q_{n,h}\cdot M_{n+1,h}.
\end{equation}
\end{theorem}

\begin{proof} 
 First, we prove the equality \eqref{eq:M(n)}. 

By the definition of the matrices $Q_{n,h}$, $h=1,\dots,s_{n+1},$ we have
\begin{align*}
\sum_{h=1}^{s_{n+1}}q^{n,h}_{i,j}&=\sum_{h=1}^{s_{n+1}} \sum_{W \in \p_i({\cal W}_{n})} \mathbf{1}_{\phi_{n+1}(W)={\cal A}_{n+1,h}}\cdot \mathbf{1}_{\ty(W)=j}\\
&= \sum_{W \in \p_i({\cal W}_{n})}\mathbf{1}_{\ty(W)=j} \sum_{h=1}^{s_{n+1}} \mathbf{1}_{\phi_{n+1}(W)={\cal A}_{n+1,h}}\\
&=\sum_{W \in \p_i({\cal W}_{n})}\mathbf{1}_{\ty(W)=j} =m^{(n)}_{i,j}.
\end{align*}
 Now, in order to prove the formula \eqref{eq:M(n+1)}, let  us start by computing, for $(i,j)\in {\mathcal J}\times {\mathcal J },$ the entry in the row $i$ and column $j$ of the $6 \times 6$ matrix 
$\sum_{h=1}^{s_{n+1}}Q_{n,h}M_{n+1,h} $, i.e. the number 
$$ \sum_{h=1}^{s_{n+1}}\sum_{\nu \in {\mathcal J}}q^{n,h}_{i,\nu}\cdot m^{n+1,h}_{\nu,j}.$$
By  the definition of the counting matrix given in formula \eqref{eq:defmatrixQ} and the definition of $m^{n+1,h}_{\nu,j}$, that  can be expressed by the equation  $$m^{n+1,h}_{\nu,j}=\sum_{\overline{W} \in \p_{\nu} ({\mathcal A}_{n+1,h})} \mathbf{1}_{\ty({\overline W})=j}, \mbox{ for } \nu,j \in {\mathcal J}, n\ge 1, 1\le h \le s_{n+1},$$

\noindent the above double sum equals
$$\sum_{h=1}^{s_{n+1}}\sum_{\nu \in {\mathcal J}}\sum_{W \in \p_i({\mathcal W}_n)}\mathbf{1}_{\phi_{n+1}(W)={\mathcal A}_{n+1,h}}\cdot \mathbf{1}_{\ty(W)=\nu} \sum_{\overline{W} \in \p_{\nu}({\mathcal A}_{n+1,h})} \mathbf{1}_{\ty(\overline{W})=j} .$$

Since, for all $W \in \p_i({\mathcal W}_n)$,
$$\sum_{\nu \in {\mathcal J}}\mathbf{1}_{\ty(W)=\nu}\sum_{\overline{W} \in \p_{\nu}({\mathcal A}_{n+1,h})} \mathbf{1}_{\ty(\overline{W})=j} =m^{n+1,h}_{\ty(W),j},$$
it follows that the above quadruple sum equals 
$$\sum_{h=1}^{s_{n+1}} \sum_{W \in \p_i({\mathcal W}_n)}\mathbf{1}_{\phi_{n+1}(W)={\mathcal A}_{n+1,h}} \cdot m^{n+1,h}_{\ty(W),j} =m^{(n+1)}_{i,j},$$
by the definition of $M(n+1)$ and the construction methods of paths in labyrinth sets of level $n\ge 2.$ 
\end{proof}
\begin{remark}  
For $s_{n+1}=1$, (for some $n\ge 1$), we have $M(n)=Q_{n,1}$ and thus in this
case we recover the formula $M(n+1)=M(n)\cdot M_{n+1}$ that holds for mixed
labyrinth fractals \cite{mixlaby}.  Moreover, let us recall here, that
throughout our consideration we have $s_1=1$, i.e., the first step of the
construction is defined by exactly one pattern.
\end{remark}

From the above theorem we immediately obtain the following recursion equation for the counting matrices.
\begin{corollary}
\label{cor:counting_matrices}
With the above notation,
 we define $Q_{0,1}:= I$, the identity $6\times 6$-matrix, and for all $n\ge 1$ we have 
\begin{equation}
\label{eq:counting_matrices_recursion2}
Q_{n,1}+\dots + Q_{n,s_{n+1}}=Q_{n-1,1}\cdot M_{n,1}+\dots +Q_{n-1, s_{n}}\cdot M_{n,s_{n}}.
\end{equation}
\end{corollary}

\begin{lemma} 
\label{lemma: paths intersection pattern}
Let ${\cal A} \subseteq {\cal S}_m $ be a labyrinth pattern (or ${\cal A}={\mathcal W}_n$ is a labyrinth set of level $n$ ) with exits $W^{\etop}, W^{\ebottom}, W^{\eleft}, W^{\eright}$. Then $p(W^{\etop},W^{\ebottom})\cap p(W^{\eleft},W^{\eright})\ne \emptyset$. \\More precisely, one of the following cases can occur:
\begin{enumerate}
\item[(a)] $p(W^{\etop},W^{\ebottom})\cap p(W^{\eleft},W^{\eright})= \{V_1  \}$, for some $V_1\in {\mathcal V}({\mathcal G}({\mathcal A}))$
\item[(b)] if there exist $V_1, V_2\in {\mathcal V}({\mathcal G}({\mathcal A}))$, $V_1 \ne V_2$, with $p(W^{\etop},W^{\ebottom})\cap p(W^{\eleft},W^{\eright})\supseteq\{V_1,V_2  \}$, then the path $p(V_1,V_2)$ in ${\mathcal G}({\mathcal A})$ satisfies $p(V_1,V_2)\subseteq p(W^{\etop},W^{\ebottom})$ and $p(V_1,V_2)\subseteq p(W^{\eleft},W^{\eright})$.
\end{enumerate}
\end{lemma}

\begin{proof}[Proof (Sketch).] Let $m$  be the width of ${\cal A}$ and let us
consider the lattice $\mathbb{L}_m:=\{1,\dots,m\}^2$. Then ${\mathcal
G}({\mathcal A})$ induces on $\mathbb{L}_m$ a subgraph $\mathbb{L}({\mathcal
A})$ in a natural way. In order to prove the first assertion of the lemma it is
enough to show that a (cycle-free) path in $\mathbb{L}({\mathcal A})$ that
connects $(x,1)$ with $(x,m)$ intersects any (cycle-free) path $\mathbb{L}_m$
that connects $(1,y)$ with $(m,y)$, where $x,y \in \{1,\dots,m  \}.$ We use
the natural embedding of $\mathbb{L}_m$ and $\mathbb{L}({\mathcal A})$ in the plane, scaled
by factor $1/m$.

From this embedding we obtain two curves in the unit square,
one leading from top to bottom and one leading from left to right,
 corresponding to the two
mentioned paths in the induced graph. By \cite{Maehara84}[Lemma 2], these
two curves intersect.

Since the intersection of the two corresponding paths in ${\mathcal
G}({\mathcal A})$ is non-empty, it can be either a single vertex, which yields
case (a), or contain more than one vertex, which leads to case (b), where the
assertion can be easily obtained by indirect proof, based on the fact that the
graph of a labyrinth pattern (or set) is a tree.  
\end{proof}

The above result easily yields the following.

\begin{lemma}
\label{lemma:sums_of_pathlengths}
In every labyrinth pattern ${\cal A} \subseteq {\cal S}_m $  (or ${\cal A}={\mathcal W}_n$  labyrinth set of level $n$ ) with exits $W^{\etop}, W^{\ebottom}, W^{\eleft}, W^{\eright}$ the lengths of the paths in ${\mathcal G}({\mathcal A})$ between exits satisfy 
\begin{align}
& \ell(p(W^{\etop},W^{\ebottom}))+\ell(p(W^{\eleft},W^{\eright}))= \nonumber \\
\max & \{  \ell(p(W^{\eleft},W^{\ebottom}))+\ell(p(W^{\etop},W^{\eright})), \nonumber \\
   &  \ell(p(W^{\eleft},W^{\etop}))+\ell(p(W^{\ebottom},W^{\eright}))
 \}
\end{align}
\end{lemma}

\section{Arcs in supermixed labyrinth fractals}\label{sec:arcs_supermixed}
The following lemma establishes a connection between paths in supermixed labyrinth sets and arcs in labyrinth fractals. It provides tools for the construction and study of arcs in the fractal. Its proof works analogously as in the case of self-similar labyrinth fractals (see, e.g., \cite[Lemma 6]{laby_4x4}) by using a theorem from Kuratowski's book \cite[Theorem 3, par. 47, V, p181]{Kuratowski}.
\begin{lemma}[Arc construction]
\label{lemma:arc_construction}
Let $a,b\in L_{\infty}$, with $a\neq b$. For all $n \ge 1$, there are $W_n(a),W_n(b)\in {\mathcal V}(\mathcal{G}({\cal W}_{n}))$ such that 
\begin{itemize}
\item[(a)]$W_1(a)\supseteq W_2(a)\supseteq\ldots$,
\item[(b)]$W_1(b)\supseteq W_2(b)\supseteq\ldots$,
\item[(c)]$\{a\}=\bigcap_{n=1}^{\infty}W_n(a)$,
\item[(d)]$\{b\}=\bigcap_{n=1}^{\infty}W_n(b)$.
\item[(e)]The set $\bigcap_{n=1}^{\infty}\left(\bigcup_{W\in p_n(W_n(a),W_n(b))} W\right)$ is an arc between $a$ and $b$ in $L_{\infty}$. 
\end{itemize}
\end{lemma}

\vspace{0.2cm}

The proof of the following two propositions works analogously to the case of mixed labyrinth fractals \cite{mixlaby}.

\begin{proposition}\label{lemma:length_of_parametrisation} Let $n,k\ge 1$, $\{W_1,\ldots,W_k\}$ be a (shortest) path between the exits $W_1$ and $W_k$ in 
$\mathcal{G}({\cal W}_{n})$,  $K_0=W_1 \cap \fr([0,1]\times[0,1])$, $K_k=W_k \cap \fr([0,1]\times[0,1])$, and $c$ be a curve in $L_n$ 
from a point of $K_0$ to a point of $K_k$. The length of any parametrisation of $c$ is at least $(k-1)/(2\cdot m(n))$.
\end{proposition}
\begin{proposition}\label{lemma:ArcSimilarity1} Let $e_1,e_2$ be two exits in $L_{\infty}$, and $W_n(e_1), 
W_n(e_2)$ be the exits in $\mathcal{G}({\cal W}_{n})$ of the same type as $e_1$ and $e_2$, respectively, 
for some $n\ge 1.$
If $a$ is the arc that connects $e_1$ and $e_2$ in $L_{\infty}$, $p$ is the path in 
$\mathcal{G}({\cal W}_{n})$ from $W_n(e_1)$ to $W_n(e_2)$, and $W\in {\cal W}_{n}$ is a $\A$-square 
with respect to $p$, then $W\cap a$ is an arc in $L_{\infty}$ between the top and the bottom exit of $W$. 
If $W$ is an other type of square, the corresponding analogous statement 
holds.
\end{proposition}

Based on an alternative definiton of the box counting 
dimension 
\cite[Definition 1.3.]{Falconer_book} and making use of the property that 
one can use, instead of $\delta \to 0$, appropriate sequences $(\delta_k)_{k\ge 0}$ 
for the computation of the box counting dimension (see the above reference), one can show that the following result also holds in the context of supermixed labyrinth fractals.

\begin{proposition}\label{prop:Boxcounting} If $a$ is an arc between the top and the bottom exit in the supermixed labyrinth fractal $L_{\infty}$ then 
\[
\liminf_{n\rightarrow \infty} \frac{\log(\A(n))}{\sum_{k=1}^{n}\log (m_k)}= \underline{\dim}_B(a)\le\overline{\dim}_B(a)= \limsup_{n\rightarrow \infty} \frac{\log(\A(n))}{\sum_{k=1}^{n}\log (m_k)}.
\]

For the other pairs of exits, the analogous statement holds. 
\end{proposition}

\begin{lemma}\label{lemma:arcs intersect}
Let $L_{\infty}$ be a supermixed labyrinth fractal with the top, bottom, left and right exit denoted by $t_{\infty}, b_{\infty}, l_{\infty}, r_{\infty}$, respectively. Then the arcs  $a(t_{\infty}, b_{\infty})$ and  $a(l_{\infty}, r_{\infty})$ in $L_{\infty}$ have non-empty intersection.
\end{lemma}

\begin{proof}[Proof (idea).] One easy way to prove this is to apply \cite{Maehara84}[Lemma 2] to the arcs connecting exits of the supermixed labyrinth fractal that lie on opposite sides of the unit square.
\end{proof}
With the above notations, we have:
\begin{lemma} \label{lemma:arcs intersection}
$a(t_{\infty}, b_{\infty})\cap a(l_{\infty}, r_{\infty})$ is either a point or a common subarc of $a(t_{\infty}, b_{\infty})$ and  $a(l_{\infty}, r_{\infty})$ in the supermixed labyrinth fractal $L_{\infty}$.
\end{lemma}

\noindent \begin{proof} We sketch the proof of a more general version of the lemma. We show that if $a_1, a_2$ are arcs in a dendrite $D$ and $a_1 \cap a_2 \ne \emptyset $ then $a_1 \cap a_2 $ is a point or a (proper) arc such that $a \subseteq a_1$ and $a \subseteq a_2$. 

 Since $a_1 \cap a_2 \ne \emptyset$, there exists an $x\in D$ such that $x\in a_1 \cap a_2 $. Assume now, there exists $y\in D,$  $y \ne x,$ such that $y\in a_1 \cap a_2 $. This means that there exist  arcs, $a_1(x,y) \subseteq a_1$ and $a_2(x,y) \subseteq a_2$, that connect $x$ and $y$ in the dendrite. Since the arc between any two distinct points in a dendrite is unique \cite[Theorem~3, par. 47, V, p. 181]{Kuratowski}, we must have $a_1(x,y)=a_2(x,y)= a(x,y)$, $a(x,y)\subseteq a_1$, $a(x,y)\subseteq a_2.$

\end{proof}

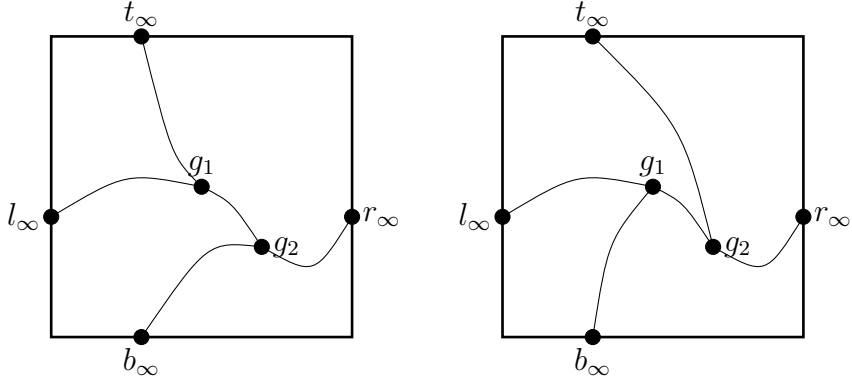
\begin{figure}[h!]
\begin{center}
\begin{tikzpicture}[scale=.40]

\draw[line width=1pt] (0,0) rectangle (10,10);
\draw[fill] (0,4) circle [radius=0.25];
\draw[fill] (10,4) circle [radius=0.25];
\draw[fill] (5,5) circle [radius=0.25];
\draw[fill] (7,3) circle [radius=0.25];
\draw[fill] (3,0) circle [radius=0.25];
\draw[fill] (3,10) circle [radius=0.25];
\coordinate[label=left:$l_{\infty}$] (linf1) at (0,4);
\coordinate[label=above:$t_{\infty}$] (tinf1) at (3,10);
\coordinate[label=below:$b_{\infty}$] (binf1) at (3,0);
\coordinate[label=right:$r_{\infty}$] (rinf1) at (10,4);
\coordinate[label=above:$g_1$] (g11) at (5,5);
\coordinate[label=right:$g_2$] (g21) at (7,3);

\path[draw](0,4) .. controls (2.5,5.5) .. (5,5);
\path[draw](5,5) .. controls (6,4.5) .. (7,3);
\path[draw](7,3) .. controls (8.75,2) .. (10,4);
\path[draw](3,10) .. controls (4,6) .. (5,5);
\path[draw](3,0) .. controls (5.2,3.2) .. (7,3);

\draw[line width=1pt] (15,0) rectangle (25,10);
\draw[fill] (15,4) circle [radius=0.25];
\draw[fill] (25,4) circle [radius=0.25];
\draw[fill] (20,5) circle [radius=0.25];
\draw[fill] (22,3) circle [radius=0.25];
\draw[fill] (18,0) circle [radius=0.25];
\draw[fill] (18,10) circle [radius=0.25];
\coordinate[label=left:$l_{\infty}$] (linf2) at (15,4);
\coordinate[label=right:$r_{\infty}$] (rinf2) at (25,4);
\coordinate[label=above:$t_{\infty}$] (tinf2) at (18,10);
\coordinate[label=below:$b_{\infty}$] (binf2) at (18,0);
\coordinate[label=above:$g_1$] (g12) at (20,5);
\coordinate[label=right:$g_2$] (g22) at (22,3);

\path[draw](15,4) .. controls (17.5,5.5) .. (20,5);
\path[draw](20,5) .. controls (21,4.5) .. (22,3);
\path[draw](22,3) .. controls (23.75,2) .. (25,4);
\path[draw](18,10) .. controls (21,7) .. (22,3);
\path[draw](18,0) .. controls (18.5,3.2) .. (20,5);

\end{tikzpicture}

\caption{Relative positions of the exits in $L_{\infty}$ and the points $g_1,g_2$: the two cases described in Lemma \ref{lemma:configurations}  }
\label{fig:configurations}
\end{center}
\end{figure}

The above lemma, and simple combinatorial arguments lead to the following result (see also Figure \ref{fig:configurations}).

\begin{lemma}\label{lemma:configurations}
 Let $t_{\infty}, b_{\infty}, l_{\infty}, r_{\infty}$ be the exits of a supermixed labyrinth fractal $L_{\infty}$. If there exist $g_1, g_2$  in $L_{\infty}$ such that $g_1 \ne g_2$ and $g_1,g_2 \in a(t_{\infty}, b_{\infty})\cap a(l_{\infty}, r_{\infty})$, then, the following positions of the points $g_1$ and $g_2$ with respect to the four exits are possible (up to symmetry):
\begin{enumerate}
\item[(a)] $g_1$ separates the points $t_{\infty}$ and $g_2$ on the arc $a(t_{\infty}, b_{\infty})$ and $g_1$ separates the points $l_{\infty}$ and $g_2$ on the arc $a(l_{\infty}, r_{\infty})$.
\item[(b)] $g_2$ separates the points $t_{\infty}$ and $g_1$ on the arc $a(t_{\infty}, b_{\infty})$ and $g_1$ separates the points $l_{\infty}$ and $g_2$ on the arc $a(l_{\infty}, r_{\infty})$.
\end{enumerate}
\end{lemma}

The following proposition is an immediate consequence of the facts in the Lemmas \ref{lemma:arcs intersection} and \ref{lemma:configurations}.
\begin{proposition} \label{lemma:sums_of_arclengths}
Let $L_{\infty}$ be a supermixed labyrinth fractal with the top, bottom, left and right exit denoted by $t_{\infty}, b_{\infty}, l_{\infty}, r_{\infty}$, respectively. Then 
\begin{align}
&\ell(a(t_{\infty},b_{\infty}))+\ell(a(l_{\infty},r_{\infty}))= \nonumber \\ 
 &\max \{\ell(a(t_{\infty},l_{\infty}))+\ell(a(b_{\infty},r_{\infty})),\ell(a(t_{\infty},r_{\infty}))+\ell(a(l_{\infty},b_{\infty})) \}.
\end{align}

\end{proposition}

\section{Blocked labyrinth patterns}\label{sec:blocked}
An $m\times m$-labyrinth pattern ${\cal A}$ is called \emph{horizontally blocked} if the row (of squares) 
from the left to the right exit contains at least one black square. It is called \emph{vertically blocked} if the 
column (of squares) from the top to the bottom exit contains at least one black square. Analogously we define for any 
$n \ge 1$ a horizontally or vertically blocked labyrinth set of level $n$.
One can easily check that horizontally or 
vertically blocked $m\times m$-labyrinth patterns only exist for for $m\ge 4$. For example, the labyrinth patterns shown in Figure \ref{fig:A1tildeA2} 
are horizontally and vertically blocked, while those in Figure \ref{fig:not_blocked} are not blocked. 

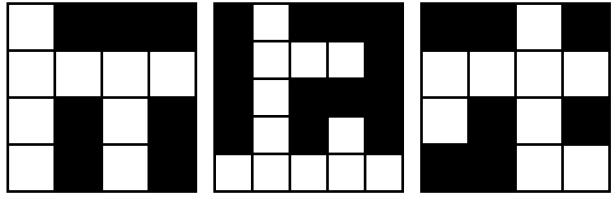
\begin{figure}[hhhh]
\begin{center}
\begin{tikzpicture}[scale=.25]
\draw[line width=1pt] (0,0) rectangle (10,10);
\draw[line width=1pt] (2.5, 0) -- (2.5,10);
\draw[line width=1pt] (5, 0) -- (5,10);
\draw[line width=1pt] (7.5, 0) -- (7.5,10);
\draw[line width=1pt] (0, 2.5) -- (10,2.5);
\draw[line width=1pt] (0, 5) -- (10,5);
\draw[line width=1pt] (0, 7.5) -- (10,7.5);
\filldraw[fill=black, draw=black] (2.5,0) rectangle (5, 5);
\filldraw[fill=black, draw=black] (7.5,0) rectangle (10, 5);
\filldraw[fill=black, draw=black] (2.5,7.5) rectangle (10, 10);
\draw[line width=1pt] (11,0) rectangle (21,10);
\draw[line width=1pt] (13, 0) -- (13,10);
\draw[line width=1pt] (15, 0) -- (15,10);
\draw[line width=1pt] (17, 0) -- (17,10);
\draw[line width=1pt] (19, 0) -- (19,10);
\draw[line width=1pt] (11, 2) -- (21,2);
\draw[line width=1pt] (11, 4) -- (21,4);
\draw[line width=1pt] (11, 6) -- (21,6);
\draw[line width=1pt] (11, 8) -- (21,8);
\filldraw[fill=black, draw=black] (11,2) rectangle (13, 10);
\filldraw[fill=black, draw=black] (15,2) rectangle (17, 6);
\filldraw[fill=black, draw=black] (17,4) rectangle (19, 6);
\filldraw[fill=black, draw=black] (19,2) rectangle (21, 8);
\filldraw[fill=black, draw=black] (15,8) rectangle (21, 10);
\draw[line width=1pt] (22,0) rectangle (32,10);
\draw[line width=1pt] (24.5, 0) -- (24.5,10);
\draw[line width=1pt] (27, 0) -- (27,10);
\draw[line width=1pt] (29.5, 0) -- (29.5,10);
\draw[line width=1pt] (22, 2.5) -- (32,2.5);
\draw[line width=1pt] (22, 5) -- (32,5);
\draw[line width=1pt] (22, 7.5) -- (32,7.5);
\filldraw[fill=black, draw=black] (22,0) rectangle (27, 2.5);
\filldraw[fill=black, draw=black] (24.5,2.5) rectangle (27, 5);
\filldraw[fill=black, draw=black] (22,7.5) rectangle (27, 10);
\filldraw[fill=black, draw=black] (29.5,2.5) rectangle (32, 5);
\filldraw[fill=black, draw=black] (29.5,7.5) rectangle (32, 10);
\end{tikzpicture}
\caption{Examples of labyrinth patterns, that are neither horizontally nor vertically blocked}
\label{fig:not_blocked}
\end{center}
\end{figure}

Let us recall some important results  obtained in the case of self-similar and mixed labyrinth fractals constructed based on blocked labyrinth patterns.

In the self-similar case the following result was proven \cite[Theorem 3.18]{laby_oigemoan}:
\begin{theorem} 
\label{theo:main result_selfsimilar}
Let $L_{\infty}$ be the (self-similar) labyrinth fractal generated by a horizontally and 
vertically blocked labyrinth pattern of width $m$. 
with path matrix $M$ and $r$ be the spectral radius of $M$.
\begin{enumerate}
\item[(a)]Between any two points in $L_{\infty}$ there is a unique arc ${a}$.
\item[(b)]The length of ${a}$ is infinite.
\item[(c)]The set of all points, at which no tangent to ${a}$ exists, 
is dense in ${a}$.
\item[(d)]$\dim_B(a)=\frac{\log(r)}{\log(m)}$.
\end{enumerate}
\end{theorem}
For the case of mixed labyrinth fractals, the following two results were proven in \cite{cristealeobacher_arcs}.
\begin{theorem}
\label{theo:existfinitearcs_mixlaby}
There exist  sequences $\{{\cal A}_k\}_{k=1}^{\infty}$ of (both horizontally and vertically) blocked labyrinth patterns, such that the limit set $L_{\infty}$ has the property that
 for any two points in $L_{\infty}$ the length of the arc $a\subset L_{\infty}$ that connects them is finite. For almost all points $x_0 \in a$ (with respect to the length) there exists the tangent at $x_0$ to the arc $a$.
\end{theorem}

\begin{proposition}
\label{prop:existinfinitearcs_mixlaby}
There exist  sequences $\{{\cal A}_k\}_{k=1}^{\infty}$ of (both horizontally and vertically) blocked labyrinth patterns, such that the limit set $L_{\infty}$ has the property that
 for any two points in $L_{\infty}$ the length of the arc $a\subset L_{\infty}$ that connects them is infinite.
\end{proposition}

At this point, we state a conjecture for the supermixed labyrinth fractals. 

\begin{conjecture}
\label{conj:infinite2}
Let $\{\tilde{\mathcal A}_k   \}_{k\ge 1}$ be a sequence of collections of
labyrinth patterns, such that for all $k\ge 1$ all patterns in $\tilde{\mathcal
A}_k$ are horizontally and vertically blocked, 
and the width sequence $\{m_k \}_{k\ge 1}$ satisfies $\sum_{k=1}^\infty \frac{1}{m_k}=\infty$. Then the supermixed labyrinth fractal
$L_{\infty}$ generated by the sequence $\{\tilde{\mathcal A}_k   \}_{k\ge 1}$
has the property that between any two distinct points in the fractal, the arc
that connects them in the fractal has infinite length.  
\end{conjecture}

In the next section
we give a proof of the result for the mixed case and we will highlight
some of the difficulties that occur when attempting to generalise 
the method to the supermixed case.\\

In the remainder of this section we recall some facts from \cite{laby_oigemoan} about 
path matrices of blocked labyrinth patterns.

\begin{lemma}[Lemma 3.3 in \cite{laby_oigemoan}]\label{lemma:ems33}
Let $M=(m_{i,j})_{i,j\in {\cal J}}$ be the path matrix of a blocked labyrinth pattern.
Then 
\begin{enumerate}
\item[(a)] $m_{\A,\C}=m_{\A,\E}$ and $m_{\A,\D}=m_{\A,\F}$.
\item[(b)] $m_{\B,\C}=m_{\B,\E}$ and $m_{\B,\D}=m_{\B,\F}$.
\item[(c)] $m_{\C,\C}+m_{\E,\C}=m_{\C,\E}+m_{\E,\E}$ and 
$m_{\C,\D}+m_{\E,\D }=m_{\C,\F}+m_{\E,\F}$.
\item[(d)] $m_{\D,\C}+m_{\F,\C}=m_{\D,\E}+m_{\F,\E}$ and 
$m_{\D,\D}+m_{\F,\D}=m_{\D,\F}+m_{\F,\F}$.
\end{enumerate}
\end{lemma}

\begin{lemma}[Lemma 3.4 in \cite{laby_oigemoan}]\label{lemma:ems34}
Let $M=(m_{i,j})_{i,j\in {\cal J}}$ be the path matrix of a blocked labyrinth pattern. Then
\begin{align*}
m_{\A,j}\ge 1 \text{ for } j\in\{\A,\C,\D,\E,\F\}
\text{ and }\;m_{\B,j}\ge 1 \text{ for } j\in\{\B,\C,\D,\E,\F\}\,.
\end{align*}
\end{lemma}

\begin{lemma}[Lemma 3.5 in \cite{laby_oigemoan}]\label{lemma:ems35}
Let $M=(m_{i,j})_{i,j\in {\cal J}}$ be the path matrix of a blocked labyrinth
pattern.  Then 
\begin{align*}
m_{\C,\A}\ge 1 &\text{ or }\; m_{\E,\A}\ge 1 \,,\\
m_{\D,\A}\ge 1 &\text{ or }\; m_{\F,\A}\ge 1 \,,\\
m_{\C,\B}\ge 1 &\text{ or }\; m_{\E,\B}\ge 1 \,,\\
m_{\D,\B}\ge 1 &\text{ or }\; m_{\F,\B}\ge 1 \,.
\end{align*}
\end{lemma}

Since any labyrinth set of level $n\ge 1$ can be viewed as a labyrinth pattern,
the above lemma also holds for any (blocked) labyrinth set of some level $n\ge 1$.

\section{Arcs of infinite length: the mixed case }\label{sec:arcs_mixed}

In this section, we prove the following theorem, which is a special case of 
Conjecture \ref{conj:infinite2} and which  solves the conjecture on mixed 
labyrinth fractals formulated in recent work \cite{cristealeobacher_arcs}.

\begin{theorem}
\label{theorem:solve_conj_mixlaby}
Let $\{{\mathcal A}_{k}\}_{k\ge 1}$ be a sequence of horizontally and vertically blocked labyrinth patterns, such that the corresponding sequence of widths $\{m_k \}_{k\ge 1}$ satisfies the condition $\sum_{k\ge 1} \frac{1}{m_k}= \infty.$ Then, for all $x,y \in L_{\infty}$ with $x \ne y$ the arc in $L_{\infty}$ that connects $x$ and $y$ has infinite length.
\end{theorem}

For this we recall the concept of a reduced path matrix of a labyrinth pattern (or set) as defined in 
\cite{laby_oigemoan}.

Let ${\cal A}\subset {\mathcal S}_m$ be a labyrinth pattern and $M$ its path
matrix. Then the reduced path matrix $\overline{M}$ of ${\cal A}$ is a $4
\times 4$ matrix, obtained from $M$ as follows: first, we add the fifth row to
the third row of $M$, and then we add the sixth row to the fourth row of $M$
und thus obtain a $4 \times 6$ matrix, from which we delete the last two
columns to obtain $\overline{M}.$ Since any labyrinth set of some
level $n$ (mixed or supermixed or generated by only one labyrinth pattern) can
be viewed as a labyrinth pattern, the definition of a reduced path matrix works
for any  labyrinth set of some level $n\ge 1$. According to the way it is
constructed, the rows of a reduced path matrix are indexed by $\A$, $\B$,
$\CE$, $\DF$ (in this order), and its columns are indexed by $\A$, $\B$, $\C$,
$\D$ (in this order).

\begin{lemma}\label{lemma:reduced_path_matrix}
Let $M_1,M_2$ be two path matrices of labyrinth patterns. Then
\[
\overline{M_1 \cdot M_2}=\overline M_1 \cdot \overline M_2.
\]
\end{lemma}
\begin{proof}
In order to prove this, we use the definitions of the path matrix and the
reduced path matrix of a labyrinth pattern. Therefor, let
$M_1=(m_{i,j}^1)_{i,j \in {\mathcal J}}$ and $M_2=(m_{i,j}^2)_{i,j \in
{\mathcal J}}$ be path matrices of labyrinth patterns, $\overline M_1=(\overline
m_{i,j}^1)_{i,j \in {\overline {\mathcal J}\times {\mathcal J'}}}$, $\overline M_2=(\overline
m_{i,j}^2)_{i,j \in {\overline {\mathcal J}\times {\mathcal J'}}}$  and 
$\overline{M_1\cdot M_2}
=(\overline{m}_{i,j})_{i,j \in {\overline {\mathcal J}\times {\mathcal J'}}}$, where
$\overline {\mathcal J}=\{\A,\B, \CE, \DF  \}$ and ${\mathcal J'}=\{\A,\B, \C, \D
\}$ are sets of indices. The idea is to prove that for all $(i,j) \in  \overline
{\mathcal J}\times {\mathcal J'}$  
\begin{equation}\label{eq:reduced_path_matrix}
\overline m_{i,j}= \sum_{(r,r')\in \{ (\A,\A),(\B,\B), (\C, \CE), (\D, \DF)\} }\overline m^1_{i,r} \cdot \overline m^2_{r',j}.
\end{equation}
Due to symmetry, it is enough to consider four cases corresponding, e.g., to $(i,j)\in \{ (\A,\A), (\A, \C), (\CE, \A), (\CE, \C) \}$. 
For all these cases \eqref{eq:reduced_path_matrix}
is obtained from the definition of products of matrices,
by applying the results from Lemma \ref{lemma:ems33}.  
\end{proof}

\begin{remark} The obove result can also be proven by using the following arguments.
From Lemma \ref{lemma:ems33} 
we know that any path matrix $M$ is of the form
\[   
M=
\begin{pmatrix}
m_{11}& m_{12}& m_{13}& m_{14}& m_{13}& m_{14} \\
m_{21}& m_{22}& m_{23}& m_{24}& m_{23}& m_{24} \\
m_{31}& m_{32}& m_{33}& m_{34}& m_{33}+ a& m_{34}+b\\ 
m_{41}& m_{42}& m_{43}& m_{44}& m_{43}+ c& m_{44}+d  \\
m_{51}& m_{52}& m_{53}& m_{54}& m_{53} - a& m_{54}-b \\
m_{61}& m_{62}& m_{63}& m_{64}& m_{63}-c & m_{64}-d 
\end{pmatrix}\,,
\]
with integers $m_{i,j}$ $i,j\in \{1,\dots,6\}$ and $a,b,c,d$.
Further note from the definition of the reduced path matrix that 
$\overline M=P_\ell M P_r$ with 
\begin{align*}
P_\ell:=
\begin{pmatrix}
1&0&0&0&0&0\\
0&1&0&0&0&0\\
0&0&1&0&1&0 \\
0&0&0&1&0&1
\end{pmatrix}\;,\quad
P_r:=
\begin{pmatrix}
1&0&0&0\\
0&1&0&0\\
0&0&1&0 \\
0&0&0&1 \\
0&0&0&0 \\
0&0&0&0 
\end{pmatrix}\,.\end{align*}

Now it is straightforward to verify (for example using a computer
algebra system) that 
$P_\ell M_1 M_2 P_r=P_\ell M_1 P_r P_\ell M_2 P_r$
for any two matrices $M_1,M_2$ having this form.
\end{remark}
Lemma \ref{lemma:reduced_path_matrix} yields a simple formula for the reduced path matrix of mixed
labyrinth sets:
\begin{corollary}\label{corollary:reduced-product}
The reduced path matrix of a mixed labyrinth set of level $n$ is the product of the reduced path matrices of the patterns that define it, i.e., 
\begin{equation}\label{formi:reduced path matrices}
\overline{M(n)}= \prod_{k=1}^n \overline M_k,
\end{equation}
where $M(n)=\prod_{k=1}^n  M_k$, and $M_k$ is the path matrix of the pattern ${\cal A}_k$ that defines the $k$-th step of the construction, $k \ge 1$.
\end{corollary}

For any integer $m\ge 4$, 
we define the {\em virtual reduced path matrix} 
\[
L(m):=
\begin{pmatrix}
m-2&      0&      1&     1\\
    0&  m-2&      1&     1\\
    1&       1&  m-1&  0 \\
    1&       1&      0&m-1
\end{pmatrix}\,.
\]

We use the attribute ``virtual'' because there need not exist a pattern to
which a given virtual reduced path matrix $L(m)$ corresponds.
The significance of these matrices is that they provide us with lower bounds
on ``real'' reduced path matrices.

\begin{proposition}
\label{prop:basic_ineq_virtual_matrix}
Let $M$ be the matrix of a horizontally and vertically blocked labyrinth
pattern of width $m\ge 4$ and $L(m)$ the virtual reduced path matrix with
parameter $m$. Furthermore, let $0<c<1$.  

Then we have, elementwise,
\begin{equation}
\label{eq:Mbar_inequality}
\overline M(1,1,1+c,1+c)^{\etop}\ge L(m)(1,1,1+c,1+c)^{\etop}.
\end{equation}
\end{proposition}

\begin{proof}
Due to symmetry it is enough to prove the inequalities for the first and third entry of the vectors resulting from this multiplication.

We introduce some temporary notation: for squares
along a path from one exit to another one 
we discriminate between the directions in which the square may be passed.
Thus we introduce the corresponding ``oriented'' types of squares
$\Aou,\Auo,\Blr,\Brl,\Cor,\Cro,\Dru,\Dur,\Eul,\Elu,\Flo,\Fol$,
and we write $m_{\Aou,\Aou}$ for the number of $\Aou$-squares in the path from 
top to bottom and similar for the other symbols.

Given a path from the top exit to the bottom exit, i.e., a path of type 
\Aou{}, we count the directed square types along the path.
Starting at the top exit, each square of type 
\Aou{} displaces
us by one unit down, zero units to the left or right.
Each square of type \Cor{} displaces
us by 1/2 unit down, 1/2 units to the right. Whatever the shape of the path,
finally we must have a total displacement by $m$ units down and zero units to the left or right. In other words, 
we may assign a 2-dimensional vector to each of the square types, 
where $v(\Aou{})=(0,-1)$,
$v(\Cor{})=(\frac{1}{2},-\frac{1}{2}) $ and so on. 
Of course, this assignment is not injective. 

Let $p=1/\log_2(1+c)$, 
and define $||(x,y)||_p:=(|x|^p+|y|^p)^{1/p}$, for every 
vector with real entries $(x,y)$.
Note that for this choice of $p$  we have $\|(1,1)\|_p=1+c$.
\\

Consider a path from top to bottom, consisting of the squares $s_1,\ldots,s_n$. 
Since the total displacement from
top to bottom is $(0,-m)$, we need to have 
$\sum_{k=1}^n v(s_k)=(0,-m)$. 

The first entry of $\overline{M}(1,1,1+c,1+c)^\top$ equals 
\begin{align*}
\lefteqn{m_{\A,\A}+m_{\A,\B}+(1+c)(m_{\A,\C}+m_{\A,\D})}\\
&= m_{\A,\A}+m_{\A,\B}+\frac{1}{2}(1+c)(m_{\A,\C}+m_{\A,\D}+m_{\A,\E}+m_{\A,\F})\\
&= m_{\Aou,\Aou}+m_{\Aou,\Blr}+\frac{1}{2}(1+c)(m_{\Aou,\Cor}+m_{\Aou,\Dru}+m_{\Aou,\Elu}+m_{\Aou,\Fol})\\
&\quad+m_{\Aou,\Auo}+m_{\Aou,\Brl}+\frac{1}{2}(1+c)(m_{\Aou,\Cro}+m_{\Aou,\Dur}+m_{\Aou,\Eul}+m_{\Aou,\Flo})\\
&=\sum_{k=1}^n\|v(s_k)\|_p\,.
\end{align*}
(Here we have used that $m_{\A,\C}=m_{\A,\E}$ and $m_{\A,\D}=m_{\A,\F}$, from Lemma \ref{lemma:ems33}.)

From Lemma \ref{lemma:ems34} we know that the path of type $\A$ in a blocked
labyrinth pattern (or set) has to contain at least one square of 
type $\C$. Now, since the path is ``oriented'', it 
induces an orientation on those squares, so in the path we have 
at least one occurence of $\Cor$ or of $\Cro$. 
That is, there exists $k_1\in \{1,\ldots,n\}$ such that 
$\ty(s_{k_1})\in\{\Cor,\Cro\}$.
In the same way, there exist $k_2,k_3,k_3\in \{1,\ldots,n\}$
such that 
$\ty(s_{k_2})\in\{\Dur,\Dru\}$,
$\ty(s_{k_3})\in\{\Elu,\Eul\}$,
$\ty(s_{k_4})\in\{\Flo,\Fol\}$.

Note that $\sum_{k=1}^nv(s_k)=(0,-m)$ and that 
$\sum_{k\in\{k_1,k_2,k_3,k_4\}}\|v(s_{k})\|_p=2(1+c)$.
Furthermore, it is straightforward to check that 
\[
\sum_{k\in\{k_1,k_2,k_3,k_4\}}v(s_{k})\in 
\{(0,0),\pm(0,2),\pm(2,0),\pm(1,1),\pm(1,-1)\}
\]
and thus 
\(
\Big\|\sum_{k\in\{k_1,k_2,k_3,k_4\}}v(s_{k})\Big\|_p \le 2\,.
\)
By the triangle inequality for the $p$-norm,
\begin{align*}
\sum_{k=1}^n\|v(s_k)\|_p
&\ge\sum_{k\in\{k_1,k_2,k_3,k_4\}}\|v(s_{k})\|_p+\Big\|\sum_{k\notin\{k_1,k_2,k_3,k_4\}}v(s_k)\Big\|_p\\
&\ge 2(1+c)+\Big\|(0,-m)-\sum_{k\in\{k_1,k_2,k_3,k_4\}}v(s_k)\Big\|_p\\
&\ge 2(1+c)+\|(0,-m)\|-\Big\|\sum_{k\in\{k_1,k_2,k_3,k_4\}}v(s_k)\Big\|_p\\
&\ge 2(1+c)+m-2=m+2c\,.
\end{align*}
But the last expression is just the first entry of $L(m)(1,1,1+c,1+c)^\top$.

By symmetry, we get the inequality for the second entry.
\\

Next we consider the concatenation of a path from top to right with that
from left to bottom. An example of such a concatenation is illustrated in Figure \ref{fig:concatenation_tr_rb}.

\begin{figure}[h!]
\begin{center}
\begin{tikzpicture}[scale=.30]

\draw[line width=2.2pt] (0,0) rectangle (32,16);

\filldraw[fill=gray!50] (4,14) rectangle (6,16);
\filldraw[fill=gray!50] (6,10) rectangle (8,16);
\filldraw[fill=gray!50] (8,6) rectangle (10,12);
\filldraw[fill=gray!50] (10,6) rectangle (14,8);
\filldraw[fill=gray!50] (12,8) rectangle (16,10);
\filldraw[fill=gray!50] (16,8) rectangle (20,10);
\filldraw[fill=gray!50] (18,6) rectangle (22,8);
\filldraw[fill=gray!50] (20,0) rectangle (22,6);

\draw[line width=1.2pt] (2, 0) -- (2,16);
\draw[line width=1.2pt] (4, 0) -- (4,16);
\draw[line width=1.2pt] (6, 0) -- (6,16);
\draw[line width=1.2pt] (8, 0) -- (8,16);
\draw[line width=1.2pt] (10, 0) -- (10,16);
\draw[line width=1.2pt] (12, 0) -- (12,16);
\draw[line width=1.2pt] (14, 0) -- (14,16);
\draw[line width=1.2pt] (16, 0) -- (16,16);
\draw[line width=1.2pt] (18, 0) -- (18,16);
\draw[line width=1.2pt] (20, 0) -- (20,16);
\draw[line width=1.2pt] (22, 0) -- (22,16);
\draw[line width=1.2pt] (24, 0) -- (24,16);
\draw[line width=1.2pt] (26, 0) -- (26,16);
\draw[line width=1.2pt] (28, 0) -- (28,16);
\draw[line width=1.2pt] (30, 0) -- (30,16);

\draw[line width=1.2pt] (0, 2) -- (32,2);
\draw[line width=1.2pt] (0, 4) -- (32,4);
\draw[line width=1.2pt] (0, 6) -- (32,6);
\draw[line width=1.2pt] (0, 8) -- (32,8);
\draw[line width=1.2pt] (0, 10) -- (32,10);
\draw[line width=1.2pt] (0, 12) -- (32,12);
\draw[line width=1.2pt] (0, 14) -- (32,14);

\filldraw[fill=black, draw=black] (0,0) rectangle (2, 8);
\filldraw[fill=black, draw=black] (0,10) rectangle (2,16);
\filldraw[fill=black, draw=black] (2,0) rectangle (4,6);
\filldraw[fill=black, draw=black] (2,12) rectangle (4,16);
\filldraw[fill=black, draw=black] (4,8) rectangle (6,14);
\filldraw[fill=black, draw=black] (6,0) rectangle (8,4);
\filldraw[fill=black, draw=black] (6,6) rectangle (8,10 );
\filldraw[fill=black, draw=black] (8,0) rectangle (10,2 );
\filldraw[fill=black, draw=black] (8,12) rectangle (10,16 );
\filldraw[fill=black, draw=black] (10,0) rectangle (12, 2);
\filldraw[fill=black, draw=black] (10,4) rectangle (12,6 );
\filldraw[fill=black, draw=black] (10,8) rectangle (12,16 );
\filldraw[fill=black, draw=black] (12,0) rectangle (14,2 );
\filldraw[fill=black, draw=black] (12,4) rectangle (14, 6);
\filldraw[fill=black, draw=black] (12,14) rectangle (14,16 );
\filldraw[fill=black, draw=black] (14,0) rectangle (16,8);
\filldraw[fill=black, draw=black] (14,10) rectangle (16,16 );
\filldraw[fill=black, draw=black] (16,0) rectangle (18, 8);
\filldraw[fill=black, draw=black] (16,10) rectangle (18,16);
\filldraw[fill=black, draw=black] (18,0) rectangle (20,6);
\filldraw[fill=black, draw=black] (18,12) rectangle (20,16);
\filldraw[fill=black, draw=black] (20,8) rectangle (22,14);
\filldraw[fill=black, draw=black] (22,0) rectangle (24,4);
\filldraw[fill=black, draw=black] (22,6) rectangle (24,10 );
\filldraw[fill=black, draw=black] (24,0) rectangle (26,2 );
\filldraw[fill=black, draw=black] (24,12) rectangle (26,16 );
\filldraw[fill=black, draw=black] (26,0) rectangle (28, 2);
\filldraw[fill=black, draw=black] (26,4) rectangle (28,6 );
\filldraw[fill=black, draw=black] (26,8) rectangle (28,16 );
\filldraw[fill=black, draw=black] (28,0) rectangle (30,2 );
\filldraw[fill=black, draw=black] (28,4) rectangle (30, 6);
\filldraw[fill=black, draw=black] (28,14) rectangle (30,16 );
\filldraw[fill=black, draw=black] (30,0) rectangle (32,8);
\filldraw[fill=black, draw=black] (30,10) rectangle (32,16 );

\draw[line width=1pt, color=gray!20] (16,0) -- (16,16);
\draw[line width=1pt, color=gray!90] (16,8) -- (16,10);

\draw[line width=2.8pt] (5,16)--(5,14.85);
\draw[line width=2.8pt] (5,15)--(7.15,15);
\draw[line width=2.8pt] (7,15)--(7,10.85);
\draw[line width=2.8pt] (7,11)--(9.15,11);
\draw[line width=2.8pt] (9,11)--(9,6.85);
\draw[line width=2.8pt] (9,7)--(13.15,7);
\draw[line width=2.8pt] (13,7)--(13,9.15);
\draw[line width=2.8pt] (13,9)--(19.15,9);
\draw[line width=2.8pt] (19,9)--(19,6.85);
\draw[line width=2.8pt] (19,7)--(21.15,7);
\draw[line width=2.8pt] (21,7)--(21,0);

\end{tikzpicture}
\caption{The concatenation of paths in two adjacent copies of the pattern (labyrinth set of level 1): from top to right and from left to bottom}
\label{fig:concatenation_tr_rb}
\end{center}
\end{figure}

Let the combined path again consist of  ``directed squares''
$s_{k_1},\ldots,s_{k_n}$.  The third entry of $\overline{M}(1,1,1+c,1+c)^\top$
equals \begin{align*}
\lefteqn{m_{\CE,\A}+m_{\CE,\B}+(1+c)(m_{\CE,\C}+m_{\CE,\D})}\\
&= m_{\CE,\A}+m_{\CE,\B}+\frac{1}{2}(1+c)(m_{\CE,\C}+m_{\CE,\D}+m_{\CE,\E}+m_{\CE,\F})\\
&= m_{\CEou,\Aou}+m_{\CEou,\Blr}+\frac{1}{2}(1+c)(m_{\CEou,\Cor}+m_{\CEou,\Dru}+m_{\CEou,\Elu}+m_{\CEou,\Fol})\\
&\quad+m_{\CEou,\Auo}+m_{\CEou,\Brl}+\frac{1}{2}(1+c)(m_{\CEou,\Cro}+m_{\CEou,\Dur}+m_{\CEou,\Eul}+m_{\CEou,\Flo})\\
&=\sum_{k=1}^n\|v(s_k)\|_p\,.
\end{align*}

When we pass from the first to the second line in the above formul\ae, we apply the assertion 3 in Lemma \ref{lemma:ems33} c. 
By the facts in Lemma \ref{lemma:ems35}
 adapted to the case of ``oriented'' paths and squares,  it follows that at least one of the squares in a ``concatenated'' path $\CE$ is of type \Aou{} or
\Auo{}. Thus there exists $k_1\in\{1,\dots,n\}$ such that 
$\ty(s_{k_1})\in \{\Aou,\Auo\}$. In the same way there exists 
$k_1\in\{1,\dots,n\}$ such that 
$\ty(s_{k_2})\in \{\Brl,\Blr\}$.

Note that  
$\sum_{k=1}^nv(s_k)=(m,-m)$, such that $\big\|\sum_{k=1}^nv(s_k)\big\|_p=(1+c)m$, and that  $\|v(s_{k_1})\|_p+\|v(s_{k_2})\|_p=2$ and
$\|v(s_{k_1})+v(s_{k_2})\|_p=1+c$.

Thus, by the triangle inequality for the $p$-norm,
\begin{align*}
\sum_{k=1}^n\|v(s_k)\|_p
&\ge\|s_{k_1}\|_p+\|s_{k_2}\|_p+\Big\|\sum_{k\notin\{k_1,k_2\}}v(s_k)\Big\|_p\\
&= 2 +\Big\|\sum_{k=1}^n v(s_k)- v(s_{k_1})- v(s_{k_2})\Big\|_p\\
&\ge 2 +\Big\|\sum_{k=1}^n v(s_k)\Big\|_p- \|v(s_{k_1})+ v(s_{k_2})\|_p\\
&=2+(1+c)(m-1)\,.
\end{align*}
But the last expression equals the third entry of $L(m)(1,1,1+c,1+c)^\top$.

By symmetry, we get the inequality for the fourth entry.
This completes the proof.
\end{proof}

\begin{lemma}\label{lemma:exists_c}
For every $c\in (0,1)$  there exists
$\kappa>0$ such that
\[
L(m) (1,1,1+c,1+c)^\top
\ge m (1+\kappa/m)(1,1,1+c,1+c)^\top,
\]
elementwise, for all $m\ge 4$.
\end{lemma}

\begin{proof}
The assertion holds true for $\kappa$ if 
\[
\begin{pmatrix}m+2c\\m+2c\\(1+c)m+(1-c)\\(1+c)m+(1-c)
\end{pmatrix}\ge m(1+\kappa/m)\begin{pmatrix}1 \\ 1 \\ 1+c \\ 1+c 
\end{pmatrix},
\]
elementwise,
i.e., if $m+2 c\ge m(1+\kappa/m)$ and $(1+c)m+(1-c)\ge m(1+\kappa/m)(1+c)$.
This is the case if $\kappa \le \min(2 c,\frac{1-c}{1+c})$. 
Thus, setting $\kappa = \min(2 c,\frac{1-c}{1+c})>0$ gives the result. 
\end{proof}

Since $\overline M$  has positive entries 
 we have, elementwise,
\begin{equation}
\label{eq:lower}
\frac{1}{m}\overline M \lambda
\ge \frac{\min_i \lambda_i}{1+c} \frac{1}{m}\overline M (1,1,1+c,1+c)^\top\,,
\end{equation}
for all $\lambda=(\lambda_1,\lambda_2,\lambda_3,\lambda_4)^\top$ with
positive real entries and all $0<c<1$.

We can choose $c$ to maximise $\min(2 c,\frac{1-c}{1+c})$. The maximiser
satisfies $2c= \frac{1-c}{1+c}$, such that $c=(\sqrt{17}-3)/4\approx 0.14$, 
and thus
$\kappa=\sup_{c\in(0,1)}\min(2 c,\frac{1-c}{1+c})=(\sqrt{17}-3)/2\approx 0.28$.
\\

Together with  \eqref{eq:lower}, we therefore get, for these values of $c$ and $\kappa$, 
\[
\frac{1}{m} \overline M\lambda\ge \frac{\min_i \lambda_i}{1+c} (1+\kappa/m)(1,1,1+c,1+c)^\top, \text{ elementwise}.
\]

\begin{lemma}
\label{lemma:infinite_paths}
Let  $A_1,A_2,\dots$ be a sequence of patterns 
with corresponding path matrices $M_1,M_2,\dots$ and 
widths $m_1,m_2,\dots$, respectively. 

If $\sum_{k\ge 1} \frac{1}{m_k}=\infty$ then we
have, elementwise:
\[
\sup _n\prod_{k=1}^n \left(\frac{1}{m_k}\overline M_k\right) 
\begin{pmatrix}1\\1\\1\\1
\end{pmatrix}=\begin{pmatrix}\infty \\ \infty \\ \infty \\ \infty 
\end{pmatrix}.
\]
\end{lemma}
\noindent
\begin{proof}
Let again $c=(\sqrt{17}-3)/4$ and $\kappa=2c$, and let $K$ be an integer, $K>1$.
Then, for every vector $\lambda= (\lambda_1,\lambda_2,\lambda_3,\lambda_4)^\top$ with positive entries,
\[\frac{1}{m_K}\overline M_K\lambda\ge \frac{\min_i \lambda_i}{1+c}(1+\kappa/m_K)(1,1,1+c,1+c)^\top\,.\]
Now,  for every $1\le k\le K-1$ we have, by Proposition \ref{prop:basic_ineq_virtual_matrix} and Lemma \ref{lemma:exists_c}, elementwise,
\[
\left(\frac{1}{m_k}\overline M_k\right)
\begin{pmatrix}1\\1\\1+c\\1+c
\end{pmatrix}
\ge \left(\frac{1}{m_k}L(m_k)\right)
\begin{pmatrix}1\\1\\1+c\\1+c
\end{pmatrix}
\ge \Big(1+\frac{\kappa}{m_k}\Big)
\begin{pmatrix}1\\1\\1+c\\1+c
\end{pmatrix},
\]
and thus we obtain, for all vectors $\lambda$ with positive entries,
\begin{align*}
\prod_{k=1}^K \left(\frac{1}{m_k}\overline M_k\right)\lambda
&\ge \frac{\min_i\lambda_i}{1+c}\prod_{k=1}^K\Big(1+\frac{\kappa}{m_k}\Big)
\begin{pmatrix}1\\1\\1+c\\1+c
\end{pmatrix}, \text{ elementwise}.
\end{align*}

Using the above inequality with $\lambda=(1,1,1,1)^\top$, we have  $\min_i\lambda_i=1$ and therefore (elementwise) for any $K\ge 1$
\[
\sup _n\prod_{k=1}^n \left(\frac{1}{m_k}\overline M_k\right) 
\begin{pmatrix}1\\1\\1\\1
\end{pmatrix}
\ge 
\frac{1}{1+c}\prod_{k=1}^K\Big(1+\frac{\kappa}{m_k}\Big)
\begin{pmatrix}1\\1\\1+c\\1+c
\end{pmatrix}.
\]
From the lemma's hypothesis $\displaystyle \sum_{k \ge 1 } \frac{1}{m_k} =\infty$ and thus, by known facts from calculus, $\sup_K \prod_{k=1}^K(1+\frac{\kappa}{m_k}) =\infty$, which, together with taking the limit $K\to \infty$, yields, elementwise, 
\[
\sup _n\prod_{k=1}^n \left(\frac{1}{m_k}\overline M_k\right) 
\begin{pmatrix}1\\1\\1\\1
\end{pmatrix}=\begin{pmatrix}\infty \\ \infty \\ \infty \\ \infty 
\end{pmatrix}.
\]
\end{proof}
\begin{proof}[Proof of Theorem \ref{theorem:solve_conj_mixlaby}.] 
By Lemma \ref{lemma:infinite_paths} and Proposition \ref{lemma:length_of_parametrisation} it follows that in $L_{\infty}$ we have 
\begin{align*}
&\ell(a(t_{\infty},b_{\infty}))=\ell(a(l_{\infty},r_{\infty}))\\
=&
\ell(a(t_{\infty},r_{\infty}))+\ell(a(l_{\infty},b_{\infty}))\\
=&
\ell(a(l_{\infty},t_{\infty}))+\ell(a(b_{\infty},r_{\infty}))=\infty
\end{align*}

The next step is to show that, under the theorem's assumptions, all arcs between exits of the labyrinth fractal have infinite length, i.e., 
\[
\ell(a(t_{\infty},r_{\infty}))=\ell(a(l_{\infty},b_{\infty}))=
\ell(a(l_{\infty},t_{\infty}))=\ell(a(b_{\infty},r_{\infty}))=\infty. 
\]

Since $\ell(a(t_{\infty},r_{\infty}))+\ell(a(l_{\infty},b_{\infty}))= \infty$, at least one of the arc lengths in the sum in infinite, therefore let us  assume w.l.o.g that  $\ell(a(t_{\infty},r_{\infty}))=\infty$.  Now, let us consider the arc $a(l_{\infty},b_{\infty})$. Due to the corner property of labyrinth patterns, this arc cannot be totally contained in only one square of level $1$, which in this case would be a square 
$W \in {\cal V}({\mathcal G}({\mathcal W}_1))$ of type $\E$ in the path in ${\mathcal G}({\mathcal W}_1)$ from the left to the bottom exit of ${\mathcal W}_1$. It follows that the path from the left exit to the bottom exit of ${\mathcal W}_1 $, contains, in addition to a square of type $\E$, at least one square of type $\A$ or $\B$. Let us assume, w.l.o.g., it contains a square of type $\A$.

 Now, let us consider $L_{\infty}'$, the mixed labyrinth fractal defined by the sequence of patterns ${\mathcal A}_1',{\mathcal A}_2',\dots $, where ${\mathcal A}_k'={\mathcal A}_{k+1}$, for all  $k\ge 1$, with width sequence $\{m_k'\}_{k\ge 1}$. We get $\sum_{k\ge 1}\frac{1}{m_k'}=\infty$ and $L_{\infty}'$ satisfies the assumtions of the theorem. On the other hand, $L_{\infty}$ is a finite union of copies (scaled by factor $1/m_1$) of $L_{\infty}'$. Moreover, the arc $a(l_{\infty}, b_{\infty})$ in $L_{\infty}$ is a union of arcs between exits in some of the mentioned copies of $L_{\infty}'$. Since one of the squares in the path of type $\E$ in ${\mathcal G}({\mathcal W}_1)$ is of type $\A$, it follows by \cite[Proposition 3]{mixlaby} or by Proposition \ref{lemma:ArcSimilarity1} applied to mixed labyrinth fractals, that the arc $a(l_{\infty}, b_{\infty})$ in $L_{\infty}$ has as subarc a copy (scaled by factor $1/m_1$) of the arc between the left and bottom exits in  $L_{\infty}'$, which has infinite length. Therefore, it follows that 
$\ell(a(l_{\infty}, b_{\infty}))=\infty$.
Analogously one can prove that $\ell(a(l_{\infty},t_{\infty}))=\ell(a(b_{\infty},r_{\infty}))=\infty$.

Finally, by arguments analogous to those used in the self-similar case \cite[Theorem 3.18]{laby_oigemoan}, it follows that the arc between any two distinct points in the fractal $L_{\infty}$ has infinite length.
\end{proof}

\begin{remark}
The condition $\sum_{k\ge 1}\frac{1}{m_k}=\infty$ in Theorem \ref{theorem:solve_conj_mixlaby} is not a necessary condition for a sequence of blocked labyrinth patterns to generate a fractal $L_{\infty}$ with the property that the length of any arc that connects distinct points in $L_{\infty}$ is infinite. This is shown, e.g., by the following example. Let $L_{\infty}$ be a self-similar labyrinth fractal generated by a horizontally and vertically blocked $m\times m$-pattern ${\mathcal A}$ . The fractal can also be viewed as being a mixed labyrinth fractal generated by the sequence of patterns $\{{\mathcal A}_k\}_{k\ge 1}$, with ${\mathcal A}_1={\mathcal A}$, and ${\mathcal A}_k={\mathcal W}_{k}({\mathcal A})$, for $k\ge 2$, where ${\mathcal W}_{k}({\mathcal A})$ denotes the labyrinth set of level $k$ obtained in the construction of the self-similar fractal $L_{\infty}$. Thus, $m_k=m^k$, for all $k\ge 1$ and thus $\sum_{k\ge 1} \frac{1}{m_k}<\infty.$
On the other hand, Theorem \ref{theo:main result_selfsimilar} applied for the self-similar labyrinth fractal $L_{\infty}$ yields that for every two distinct points   in the fractal the length of the arc that connects them has infinite length.
\end{remark}

\begin{remark}
For  supermixed  labyrinth sets we have no result
corresponding to Corollary \ref{corollary:reduced-product}.
That is, we have no ``simple'' formula for the reduced path matrix of 
supermixed labyrinth sets. This is one of the reasons why 
the method used in the proof of Theorem \ref{theorem:solve_conj_mixlaby}
does not work in the supermixed case.
\end{remark}

\begin{remark}
Note that from Lemma \ref{lemma:ems35} it follows trivially that 
$m_{\C,\A}+ m_{\E,\A}\ge 1$. We discuss a geometric consequence of this:

Let $M$ be the path matrix of the pattern $\mathcal A_2$ and let 
$M_1$ be the path matrix of the pattern $\mathcal A_1$.
Let us consider a path between exits in $\mathcal W_2$
and also the corresponding path in $\mathcal W_1$ (so that the path
in $\mathcal W_2$ lies entirely in the one in $\mathcal W_1$).

For every pair of squares  in this  path in  $\mathcal W_1$
such that one is of type 
$\C$ and the other is of type $\E$,  we know that 
there exist at least one square of type $\A$ and one square of type $\B$ in 
the path in $\mathcal W_2$.
This consideration plays an essential role in the proof of
Proposition \ref{prop:basic_ineq_virtual_matrix}, which in turn is used
in the proof of Theorem \ref{theorem:solve_conj_mixlaby}.

However, this need not hold in the supermixed case. In the following example,
the path segment that connects the top exit of the left pattern to the bottom
exit of the right pattern, substituted in adjacent squares of $\mathcal W_1$
(see figure
\ref{fig:concatenation_supermix_counterexample}) contains neither a square of type $\A$ nor a square of type $\B$.

\begin{figure}[h!]
\begin{center}
\begin{tikzpicture}[scale=.30]

\draw[line width=2.2pt] (0,0) rectangle (28,14);

\filldraw[fill=gray!50] (6,12) rectangle (8,14);
\filldraw[fill=gray!50] (8,10) rectangle (10,14);
\filldraw[fill=gray!50] (10,8) rectangle (12,12);
\filldraw[fill=gray!50] (12,6) rectangle (14,10);
\filldraw[fill=gray!50] (14,4) rectangle (16,8);
\filldraw[fill=gray!50] (16,2) rectangle (18,6);
\filldraw[fill=gray!50] (18,0) rectangle (20,4);
\filldraw[fill=gray!50] (20,0) rectangle (22,2);

\draw[line width=1.2pt] (2, 0) -- (2,14);
\draw[line width=1.2pt] (4, 0) -- (4,14);
\draw[line width=1.2pt] (6, 0) -- (6,14);
\draw[line width=1.2pt] (8, 0) -- (8,14);
\draw[line width=1.2pt] (10, 0) -- (10,14);
\draw[line width=1.2pt] (12, 0) -- (12,14);
\draw[line width=1.2pt] (14, 0) -- (14,14);
\draw[line width=1.2pt] (16, 0) -- (16,14);
\draw[line width=1.2pt] (18, 0) -- (18,14);
\draw[line width=1.2pt] (20, 0) -- (20,14);
\draw[line width=1.2pt] (22, 0) -- (22,14);
\draw[line width=1.2pt] (24, 0) -- (24,14);
\draw[line width=1.2pt] (26, 0) -- (26,14);

\draw[line width=1.2pt] (0, 2) -- (28,2);
\draw[line width=1.2pt] (0, 4) -- (28,4);
\draw[line width=1.2pt] (0, 6) -- (28,6);
\draw[line width=1.2pt] (0, 8) -- (28,8);
\draw[line width=1.2pt] (0, 10) -- (28,10);
\draw[line width=1.2pt] (0, 12) -- (28,12);

\filldraw[fill=black, draw=black] (0,0) rectangle (2, 6);
\filldraw[fill=black, draw=black] (0,8) rectangle (2,14);
\filldraw[fill=black, draw=black] (2,0) rectangle (4,4);
\filldraw[fill=black, draw=black] (2,12) rectangle (4,14);
\filldraw[fill=black, draw=black] (4,0) rectangle (6,2);
\filldraw[fill=black, draw=black] (4,6) rectangle (6,10);
\filldraw[fill=black, draw=black] (6,4) rectangle (8,12);
\filldraw[fill=black, draw=black] (8,0) rectangle (10,10);
\filldraw[fill=black, draw=black] (10,0) rectangle (12, 8);
\filldraw[fill=black, draw=black] (10,12) rectangle (12,14);
\filldraw[fill=black, draw=black] (12,0) rectangle (14,6 );
\filldraw[fill=black, draw=black] (12,10) rectangle (14, 14);

\filldraw[fill=black, draw=black] (14,0) rectangle (16,4);
\filldraw[fill=black, draw=black] (14,8) rectangle (16,14 );
\filldraw[fill=black, draw=black] (16,0) rectangle (18, 2);
\filldraw[fill=black, draw=black] (16,6) rectangle (18,14);
\filldraw[fill=black, draw=black] (18,4) rectangle (20,14);
\filldraw[fill=black, draw=black] (20,2) rectangle (22,10);
\filldraw[fill=black, draw=black] (22,4) rectangle (24,8);
\filldraw[fill=black, draw=black] (22,12) rectangle (24,14 );
\filldraw[fill=black, draw=black] (24,0) rectangle (26,2 );
\filldraw[fill=black, draw=black] (24,10) rectangle (26,14 );
\filldraw[fill=black, draw=black] (26,0) rectangle (28, 6);
\filldraw[fill=black, draw=black] (26,8) rectangle (28,14 );

\draw[line width=1pt, color=gray!20] (14,0) -- (14,14);
\draw[line width=1pt, color=gray!90] (14,6) -- (14,8);

\draw[line width=2.8pt] (7,14)--(7,12.85);
\draw[line width=2.8pt] (7,13)--(9.15,13);
\draw[line width=2.8pt] (9,13)--(9,11);
\draw[line width=2.8pt] (8.85,11)--(11.15,11);
\draw[line width=2.8pt] (11,11)--(11,9);
\draw[line width=2.8pt] (10.85,9)--(13.15,9);
\draw[line width=2.8pt] (13,9)--(13,7);
\draw[line width=2.8pt] (12.85,7)--(15.15,7);
\draw[line width=2.8pt] (15,7)--(15,5);
\draw[line width=2.8pt] (14.85,5)--(17.15,5);
\draw[line width=2.8pt] (17,5)--(17,3);
\draw[line width=2.8pt] (16.85,3)--(19.15,3);
\draw[line width=2.8pt] (19,3)--(19,1);
\draw[line width=2.8pt] (18.85,1)--(21.15,1);
\draw[line width=2.8pt] (21,1)--(21,0);

\end{tikzpicture}
\caption{The concatenation of paths in two adjacent copies of distinct labyrinth patterns (labyrinth set of level 1): from top to right and from left to bottom}
\label{fig:concatenation_supermix_counterexample}
\end{center}
\end{figure}
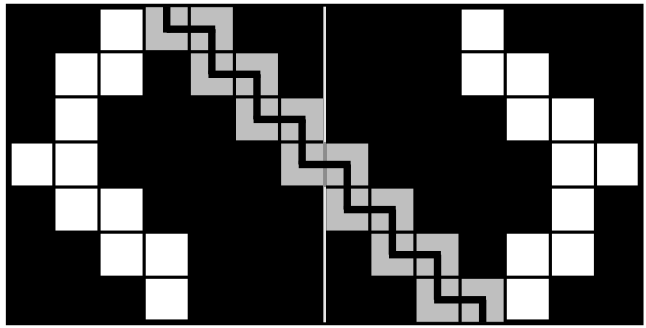
\end{remark}

\section{Conclusions}\label{sec:conclusions}

We have presented a new class of planar dendrites, termed supermixed 
labyrinth fractals, which generalises
the established concept of (mixed) labyrinth fractal.
We proved several useful results about these objects, in particular we were
able to find a recursive formula for their path matrix, 
Theorem \ref{theo:basisformel}, which makes them 
amenable to concrete computation of discrete path lengths and related 
quantities. 
In addition to generalising earlier results on the topology of (mixed) 
labyrinth fractals, 
some of our results provide new insight also into the more special cases
of self-similar and mixed labyrinth fractals. E.g., we establish important
relations between the length of arcs between exits of the fractal and 
special subarcs of the fractal.  

One of the main results is Theorem \ref{theorem:solve_conj_mixlaby}, were
we give a sufficient condition for the infinite length of any arc in a
mixed labyrinth fractal. Even though the method of proof is not strong enough
to provide the same result for the supermixed case, we believe that it also
constitutes a valuable step towards the more general result, which 
therefore still remains open.


\begin{thebibliography}{20}
\bibitem{AnhHoffmanSeegerTarafdar2005} D.\,H.\,N. Anh, K.\,H. Hoffmann, S. Seeger, S. Tarafdar, \emph{Diffusion in disordered fractals}, Europhys. Lett. {\bf 70} (1) (2005), 109--115
\bibitem{barnsley_superfractals} M.\,F. Barnsley, \emph{Superfractals},  Cambridge University Press, New York, 2006 
\bibitem{laby_4x4} L.\,L. Cristea, B. Steinsky, \emph{Curves of Infinite Length in $4\times 4$-Labyrinth Fractals}, Geometriae Dedicata, Vol. 141, Issue 1 (2009), 1--17
\bibitem{laby_oigemoan} L.\,L. Cristea, B. Steinsky, \emph{Curves of Infinite Length in Labyrinth-Fractals}, Proceedings of the Edinburgh Mathematical Society Volume 54, Issue 02 (2011), 329--344
\bibitem{mixlaby}L.\,L. Cristea, B. Steinsky, \emph{Mixed labyrinth fractals}, 
Topol. Appl.(2017), http://dx.doi.org/10.1016/jtopol2017.06.022
\bibitem{cristealeobacher_arcs}  L.\,L. Cristea, G. Leobacher, \emph{On the length of arcs in labyrinth fractals}, Monatshefte f\"ur Mathematik (2017), doi:10.1007/s00605-017-1056-8
\bibitem{Falconer_book}  K.\,J. Falconer, \emph{Fractal geometry, Mathematical Foundations and Applications}, John Wiley \& Sons, Chichester, 1990
%
\bibitem{freiberghamblyhutchinson_Vvariable} U. Freiberg, B.\,M. Hambly, J.\,E. Hutchinson, \emph{Spectral Asymptotics for $V$-variable Sierpinski Gaskets}, arXiv:1502.00711 [math.PR]
\bibitem{GrachevPotapovGerman2013} V.\,I. Grachev, A.\,A. Potapov, V.\,A. German, \emph{Fractal Labyrinths and Planar Nanostructures}, 
PIERS Proceedings, Stockholm, Sweden, Aug. 12-15, 2013 
\bibitem{tarafdar_multifractalNaCl2013} A. Giri, M. Dutta Choudhury, T. Dutta, S. Tarafdar, \emph{Multifractal Growth of Crystalline NaCl Aggregates in a Gelatin Medium},  Crystal Growth \& Design {\bf 13} (2013), 341--345
\bibitem{JanaGarcia_lithiumdendrite2017} A. Jana, R.\,E. Garcia, \emph{Lithium dendrite growth mechanisms in liquid electrolytes}, Nano Energy {\bf 41} (2017), 552-565
 \bibitem{Kuratowski} K. Kuratowski, Topology, Volume II, Academic Press, New York and London, 1968
\bibitem{MauldinWilliams_graphdirected_1988} R.\,D. Mauldin, S.\,C. Williams, \emph{Hausdorff dimension in graph directed constructions}, Trans. Amer. Math. Soc. {\bf 309} (2) (1988), 811--829
\bibitem{MauldinUrbanski_bookGDMS_2003}R.\,D. Mauldin, M.\,Urba\'nski, \emph{Graph Directed Markov Systems: geometry and dynamics of limit sets}, Cambridge Tracts in Mathematics, Cambridge University Press, 2003
\bibitem{Maehara84} R. Maehara, \emph{The Jordan Curve Theorem Via the Brouwer Fixed Point Theorem}, The American Mathematical Monthly, {\bf 91} (10)(1984), 641--643
%
%
%
%
%
\bibitem{PotapovGermanGrachev2013}  A.\,A. Potapov, V.\,A. German, V.\,I. Grachev, \emph{``Nano'' and radar signal processing: Fractal reconstruction complicated images, signals and radar backgrounds based on fractal labyrinths}, Conf. Radar Symposium (IRS) Proceedings, 2013, Vol. 2
\bibitem{PotapovZhang2016} A.\,A. Potapov, W. Zhang, \emph{Simulation of New Ultra-Wide Band Fractal Antennas Based on Fractal Labyrinths}, Proceedings of the 2016 CIE International Conference on Radar, Guangzhou, Oct 10-12, 2016, 319--323
\bibitem{PotapovPotapovPotapov_dec2017} A.\,A. Potapov, Al.\,A. Potapov, V.\,A. Potapov, \emph{Fractal radioelements, devices and fractal systems for radar and telecomunications}, Conference Paper, December 2017 (www.researchgate.net/publication/321462598)
\bibitem{Samuel_selfsimilar_dendrites} M. Samuel, A. Tetenov, D. Vaulin, \emph{Self-similar dendrites generated by polygonal systems in the plane}, Siberian Electronic Mathematical Reports, DOI 10.17377/semi.2017.14.063
\bibitem{SeegerHoffmannEssex2009_randomKoch} S. Seeger, K.\,H. Hoffmann, C. Essex, \emph{Random Walks on random Koch curves}, J. Phys. A: Math. Theor. {\bf 42} (2009) 225002 (12 pages) 
\bibitem{Tarafdar_modelporstructrepeatedSC2001} S. Tarafdar, A. Franz, C. Schulzsky, K.\,H. Hoffmann, \emph{Modelling porous structures by repeated Sierpi\'nski carpets}, Physica A {\bf 292} (2001), 1--8
\bibitem{Tricot} C. Tricot, Curves and Fractal Dimension, Springer-Verlag, Paris, 1993
%
%
%
\end{thebibliography}
\end{document}